\documentclass[11pt]{article}

\usepackage{lineno,hyperref}
\modulolinenumbers[5]

\usepackage{amsmath,amssymb,amsthm,amsfonts,afterpage,float,bm,stmaryrd,paralist,wrapfig,bbm,soul}

\usepackage{cancel}
\usepackage{color}
\usepackage{fancyhdr}
\usepackage{graphics}
\usepackage{placeins}
\usepackage{hyperref}
\usepackage{graphicx,epsf}
\usepackage{makeidx}
\usepackage{epsfig}
\usepackage{lscape}
\usepackage{tcolorbox}
\usepackage{setspace}
\usepackage{empheq,xcolor}
\usepackage[margin=1.5in]{geometry}    
\usepackage[english]{babel}
\usepackage[utf8]{inputenc}
\usepackage{algorithm}
\usepackage[noend]{algpseudocode} 
\usepackage{upgreek}

\pssilent

\newtheorem{theorem}{Theorem}[section]
\newtheorem{lemma}{Lemma}[section]

\newtheorem{definition}{Definition}[section]

\numberwithin{equation}{section}

\def\XXint#1#2#3{{\setbox0=\hbox{$#1{#2#3}{\int}$}
\vcenter{\hbox{$#2#3$}}\kern-.51\wd0}}

\newcommand{\R}{\mathbb{R}}

\DeclareMathOperator*{\argmin}{arg\,min}

\graphicspath{{Figure/}}

\providecommand{\keywords}[1]
{
  \noindent
  \textbf{Keywords.} #1
}

\begin{document}

\setlength{\pdfpageheight}{\paperheight}
\setlength{\pdfpagewidth}{\paperwidth}
\title{Supervised Gromov-Wasserstein Optimal Transport}
\author{Zixuan Cang$^{1,*}$, Yaqi Wu$^2$, Yanxiang Zhao$^{2,}$\footnote{Correspondence: zcang@ncsu.edu, yxzhao@email.gwu.edu}
 \vspace{0.1in}\\
$^1$Department of Mathematics, Center for Research in Scientific Computation \\
North Carolina State University, NC, USA \\
$^2$Department of Mathematics \\
The George Washington University, Washington D.C., USA}

\date{}
\maketitle

\begin{abstract}

We introduce the supervised Gromov-Wasserstein (sGW) optimal transport, an extension of Gromov-Wasserstein by incorporating potential infinity patterns in the cost tensor. sGW enables the enforcement of application-induced constraints such as the preservation of pairwise distances by implementing the constraints as an infinity pattern. A numerical solver is proposed for the sGW problem and the effectiveness is demonstrated in various numerical experiments. The high-order constraints in sGW are transferred to constraints on the coupling matrix by solving a minimal vertex cover problem. The transformed problem is solved by the Mirror-C descent iteration coupled with the supervised optimal transport solver. In the numerical experiments, we first validate the proposed framework by applying it to matching synthetic datasets and investigating the impact of the model parameters. Additionally, we successfully apply sGW to real single-cell RNA sequencing data. Through comparisons with other Gromov-Wasserstein variants on real data, we demonstrate that sGW offers the novel utility of controlling distance preservation, leading to the automatic estimation of overlapping portions of datasets, which brings improved stability and flexibility in data-driven applications.

\end{abstract}



\section{Introduction}

The Gromov-Wasserstein(GW) distance \cite{memoli2011gromov} can be viewed as an integration of the  Wasserstein distance and Gromov-Hausdorff distance based on the optimal transport (OT) formalism for comparing metric measure spaces. Formulated as a nonconvex quadratic optimization problem, it introduces a coupling between two measures that minimizes distance disturbance. Unlike the standard Wasserstein distance defined by OT, the GW distance is based on the idea of comparing the internal structures of the two metric measure spaces, rather than connecting two measures on the same metric space. The alignment of different metric measure spaces such as point clouds embedded in different dimensional Euclidean spaces is ubiquitous in many practical problems in machine learning and computational biology. 
The GW distance and the resulting coupling have found applications to a broad range of applications such as network analysis \cite{xu2019gromov}, graph theory \cite{petric2019got}, natural language processing \cite{alvarez2018gromov}, and computational biology \cite{demetci2022scot}. 
In practice, the problem is often formulated as comparing distance matrices from different metric spaces, where the optimal coupling may be sensitive to small structural perturbations and therefore can be difficult to interpret. Additionally, in many applications, the two objects to be coupled could overlap only partially, such as two partially overlapping datasets of the same system. However, designed as a distance between metric measure spaces, the original OT problem for computing GW distance enforces balanced constraints on the measures and does not consider partial coupling. Computing the original GW distance is also computationally challenging which leads to an NP-hard problem, limiting its applications to real applications with large datasets \cite{zheng2022brief}.

Recently, various efforts have been made to improve computational efficiency and relax the strict constraints in GW-OT to improve its practical applicability. Similar to the seminal work of the efficient Sinkhorn algorithm for OT \cite{cuturi2013sinkhorn}, by adding an entropy regularization, a regularized version of GW distance was introduced with a fast iterative algorithm for the resulting non-convex optimization problem \cite{peyre2016gromov}. Based on a recent method for non-convex optimization \cite{boct2016inertial}, a compact mapping algorithm was also introduced for the regularized GW which guarantees to find a local minimum \cite{solomon2016entropic}. Like unbalanced OT \cite{chizat2018scaling} and partial OT \cite{figalli2010optimal}, similar variants have been developed for GW-OT to relax the balanced constraints \cite{chapel2020partial}. A Frank-Wolfe-based algorithm was proposed to solve the partial GW (PGW) problem \cite{chapel2020partial}, utilizing strategies previously used in unbalanced OT and partial OT \cite{guittet2002extended,figalli2010optimal}. An unbalanced GW (UGW) was also introduced which adds two quadratic $\phi$-divergences to the objective \cite{sejourne2023unbalanced}. In UGW, a lower bound was derived in terms of tensorized entropy. In a special case when $\phi$-divergence is the Kullback-Leibler (KL) divergence, a Sinkhorn-based algorithm was derived by drawing inspiration from strategies used for unbalanced OT \cite{chizat2018scaling}. For general $\phi$, the gap between UGW and the lower bound remains rather unclear and could be sensitive to the choice of $\phi$.

While PGW and UGW feature relaxed constraints, certain applications may require other ways of relaxing the balanced constraints. For instance, prior knowledge may be unavailable to assign the total transported mass required by PGW, and the two marginal masses may differ significantly which makes it challenging to apply UGW. Having identified similar limitations in partial OT and unbalanced OT, we have previously developed a supervised OT \cite{cang2022supervised} which enforces elementwise constraints on the transport plan to exclude unreasonable coupling and has been successfully applied to a crucial practical problem in biology, the inference of cell-cell communication from spatial data \cite{cang2023screening}. Here, we aim to develop an analogous extension for GW-OT which excludes unreasonable couplings of pairs of points.

We introduce supervised Gromov-Wasserstein (sGW), a method through which we can {\it supervise} the extent to which the metric is preserved after it is mapped to the other metric space using the optimal transport plan. To this end, we set a geometric threshold on the difference between metrics and replace elements in the cost tensor by infinity when they exceed the threshold. This results in an sGW formulation as a non-convex min-max optimization problem where couplings between certain pairs of points with significantly different intra-pair distances are avoided. To make the problem computationally tractable, we convert the elementwise constraint on the 4D tensor into a constraint on the transport plan by formulating a minimal vertex cover problem. We then solve the sGW problem by progressively computing the stationary point of the entropy regularized objective function and performing each update using the generalized Sinkhorn algorithm from our previous work on supervised OT \cite{cang2022supervised}. sGW has several unique properties, including enforcing elementwise constraints and automatically determining the total transported mass. Therefore, sGW extends the capabilities of GW, PGW, and UGW by offering a flexible method for controlling the degree to which the geometry is preserved through mapping. This capability could be useful for applications that have geometric restrictions such as single-cell genomics, traffic flow and signal processing. To further improve the efficiency for practical applications, we propose a strategy to compute sGW on subsampled data and recover the full coupling by solving another supervised OT problem.

The paper is structured as follows: In Sections \ref{sec:background}, we provide a brief review of GW and the requisite OT details. Section \ref{sec:main} is dedicated to defining the proposed sGW and the development of a numerical algorithm for sGW. This algorithm incorporates a greedy minimum vertex cover solver, mirror-C descent algorithm and a generalized Sinkhorn iteration based on the supervised optimal transport \cite{cang2022supervised}. We demonstrate the utilities and properties of sGW on different data types and applications, and introduce a strategy for handling large datasets in Section \ref{sec:numerical}. Lastly, we summarize the key findings and insights drawn from sGW, along with comments on future research directions in Section \ref{sec:conclusion}.


\section{Background on Optimal Transport and Gromov-Wassertein Distance}\label{sec:background}

\subsection{Optimal Transport Theory}\label{subsection:OT}

Let $\mathcal{X}=\left\{\mathbf{x}_i\right\}_{i=1}^n$ and $\mathcal{Y}=\left\{\mathbf{y}_j\right\}_{j=1}^m$ be two point clouds that represent the source and target, respectively. Let $\Sigma_{n}=\left\{\mathbf{p}=\left(p_i\right)_i \in \mathbb{R}_{+}^n:  \allowbreak \sum_i p_i=1\right\}$ be the probability simplex, and consider a pair of probability distributions (or histograms) $(\mathbf{a}, \mathbf{b}) \in \Sigma_n \times \Sigma_m$ over $\mathcal{X}$ and $\mathcal{Y}$:
\begin{align}\label{eqn:point_clouds}
\mathbf{a}=\sum_{i=1}^n a_i \delta_{\mathbf{x}_i}, \quad \mathbf{b}=\sum_{j=1}^m b_j \delta_{\mathbf{y}_j}.
\end{align}
The standard discrete Kantorovich optimal transport problem is defined as
\begin{align}\label{eqn:OT}
L_{\mathrm{OT}}(\mathbf{a}, \mathbf{b} ; \mathbf{C})=\min _{\mathbf{P} \in \mathbf{U}(\mathbf{a}, \mathbf{b})}\langle\mathbf{P}, \mathbf{C}\rangle_F=\sum_{i, j} P_{i j} C_{i j}.
\end{align}
Here $\mathbf{P} = (P_{ij})_{ij}$ is a transport plan and $\mathbf{C} = (C_{ij})_{ij}$ is the cost matrix. The admissible set is
$$
\mathbf{U}(\mathbf{a}, \mathbf{b}) \stackrel{\text{def.}}{=}\left\{\mathbf{P} \in \mathbb{R}_{+}^{n \times m}: \mathbf{P} \textbf{1}_m=\mathbf{a},\ \mathbf{P}^{\mathrm{T}} \textbf{1}_n=\mathbf{b}\right\}
$$
where $\textbf{1}_n$ ($\textbf{1}_{n\times m}$, respectively) is the $n$-dimensional ($n\times m$-dimensional, respectively) all-one vector (matrix, respectively), and we use the following matrix-vector notation:
$$
\mathbf{P} \textbf{1}_m=\left(\sum_j P_{ij}\right)_i \in \mathbb{R}^n, \quad  \mathbf{P}^{\mathrm{T}} \textbf{1}_n=\left(\sum_i P_{ij}\right)_j \in \mathbb{R}^m.
$$
From now on, we will omit the subscript in $\mathbf{1}_m$ or $\mathbf{1}_{n\times m}$ as its dimension can be inferred from the context. The entry $C_{ij}$ in the matrix $\mathbf{C}$ characterizes the pairwise finite cost of moving one unit of mass from $\mathbf{x}_i$ to $\mathbf{y}_j$. This is a linear programming problem whose optimal solution is not necessarily unique.

When $\|\mathbf{a}\|_1 \neq \|\mathbf{b}\|_1$ and assuming $s$ is a fraction of total mass: $0 \leq s \leq \min \left(\|\mathbf{a}\|_1,\|\mathbf{b}\|_1\right)$, partial optimal transport (POT) \cite{Benamou_SISC2015} was introduced:
\begin{align}
L^{s}_{\mathrm{POT}}(\mathbf{a}, \mathbf{b} ; \mathbf{C})=\min _{\mathbf{P} \in \mathbf{U}^{s}(\le\mathbf{a}, \le\mathbf{b})}\langle\mathbf{P}, \mathbf{C}\rangle_F,
\end{align}
where the admissible set of couplings is defined as 
\begin{align}\label{eqn:pOT_feasible_set}
\mathbf{U}^s(\leq \mathbf{a}, \leq \mathbf{b})=\left\{\mathbf{P} \in \mathbb{R}_{+}^{n \times m} \mid \mathbf{P}  \mathbf{1} \leq \mathbf{a}, \mathbf{P}^{\mathrm{T}}  \mathbf{1} \leq \mathbf{b},\quad \langle\mathbf{P}, \textbf{1} \rangle_F= s \right\}.
\end{align}

Another significant variant of OT is the unbalanced optimal transport (UOT) \cite{chizat2018scaling}:
\begin{align}\label{eqn:UOT}
L_{\mathrm{UOT}}(\mathbf{a}, \mathbf{b} ; \mathbf{C})=\min _{\mathbf{P}\ge0}\langle\mathbf{P}, \mathbf{C}\rangle_F + F(\mathbf{P}\mathbf{1}) + G(\mathbf{P}^{\mathrm{T}}\mathbf{1}),
\end{align}
where $F$ and $G$ are lower-semicontinuous convex functions. A typical choice for $F$ and $G$ is the Kullback-Leibler divergence: $F(\cdot) = \text{KL}(\cdot | \textbf{a} ), G(\cdot) = \text{KL}(\cdot|\mathbf{b})$, where the $\text{KL}$ function is defined as
\[
\text{KL}(\mathbf{p}|\mathbf{q}) = \sum_{i} p_i \log\frac{p_i}{q_i} - p_i + q_i.
\]
It is worth noting that when $F = \iota_{\{\mathbf{a}\}}$ and $G = \iota_{\{\mathbf{b}\}}$, the UOT formulation (\ref{eqn:UOT}) reduces to the original OT (\ref{eqn:OT}).

\subsection{Supervised Optimal Transport}\label{subsection:sOT}

In the conventional OT formulation, the requirement for two discrete probability distributions to have equal total mass can be overly restrictive for many practical applications. While POT and UOT offer different relaxations of mass conservation, they are limited in incorporating elementwise application-induced constraints, especially when there are some infinity entries in the cost matrix $\mathbf{C}$. To address this challenge, the supervised optimal transport (sOT) framework has been developed \cite{cang2022supervised}, drawing inspiration from the theories of POT and UOT. The sOT is particularly designed to accommodate situations where $\|\mathbf{a}\|_1 \ne \|\mathbf{b}\|_{1}$, and certain costs $C_{ij}$ may be infinite. Essentially, sOT introduces a linear programming problem with supplementary sparsity constraints imposed.

Given a cost matrix $\mathbf{C}$, which includes infinity elements, we define $\Omega_{\mathbf{C}}$ as the set of indices corresponding to these infinity entries. We then denote $M_{\mathbf{C}}^+$ as the set of $n \times m$ matrices that are structured such that any element at the position of $(i,j)\in\Omega_{\mathbf{C}}$ is set to 0, while the remaining elements are nonnegative. Within this framework, the sOT problem can be formally stated as follows \cite{cang2022supervised}:
\begin{definition}
The supervised optimal transport problem is defined as the following bilevel extremization problem:
\begin{align}\label{eqn:sOT01}
    L_{\emph{sOT}}(\mathbf{a},\mathbf{b};\mathbf{C}) = \max_{s \in [0, \min \{ \|\mathbf{a}\|_{1}, \|\mathbf{b}\|_{1} \} ]} \min_{\mathbf{P} \in \mathbf{U}(\leq \mathbf{a}, \leq \mathbf{b}) \cap M_{\mathbf{C}}^+} \Big\{ \langle\mathbf{P}, \mathbf{C}\rangle_F: \langle\mathbf{P}, \mathbf{1}\rangle_F= s \Big\},
\end{align}
where the feasible set is defined as 
\[
\mathbf{U}(\le\mathbf{a}, \le\mathbf{b}) \stackrel{\emph { def. }}{=}\left\{\mathbf{P} \in \mathbb{R}_{+}^{n \times m}: \mathbf{P} \mathbf{1} \le \mathbf{a}, \ \mathbf{P}^{\mathrm{T}} \mathbf{1} \le \mathbf{b}\right\}.
\]
\end{definition}

By its definition, the sOT solution transports the maximal possible marginal density at the lowest total cost, subject to a predefined sparsity in the transport plan $\mathbf{P}\in M_{\mathbf{C}}^+$. sOT is closely related to POT, that is in the special case where the cost matrix $\mathbf{C}$ does not contain $\infty$ entries, sOT reduces to the POT problem with $s=\min\{\|\mathbf{a}\|_1,\|\mathbf{b}\|_1\}$.

An alternative definition for sOT is given by considering the blocked marginal distributions that remain untransported. Specifically, if we define the blocked marginal distributions $\boldsymbol{\upmu} \geq 0$ and $\boldsymbol{\upnu}\geq 0$ as follows:
\begin{align*}
    \mathbf{P} \textbf{1} + \boldsymbol{\upmu} = \mathbf{a}, \quad    \mathbf{P} ^{\mathrm{T}} \textbf{1} + \boldsymbol{\upnu} = \mathbf{b},
\end{align*}
we can readily observe that
\begin{align*}
   s = \langle\mathbf{P}, \textbf{1}\rangle_F=\| \mathbf{a} \|_1 - \| \boldsymbol{\upmu} \|_1 = \| \mathbf{b} \|_1 - \| \boldsymbol{\upnu} \|_1 =\frac{1}{2} \bigl[ (\| \mathbf{a} \|_1 + \| \mathbf{b} \|_1)-(\| \boldsymbol{\upmu} \|_1 + \| \boldsymbol{\upnu} \|_1) \bigr].
\end{align*}
In other terms, the total transported mass $s$ and the total blocked mass $\tau = \| \boldsymbol{\upmu} \|_1 + \| \boldsymbol{\upnu} \|_1 $ of the marginal distributions are inversely related. Consequently, maximizing $s$ is equivalent to minimizing $\tau$. Therefore, the alternative definition for sOT can be stated as:
\begin{align}\label{eqn:sOT02}
    L_{\text{sOT}}(\mathbf{a},\mathbf{b};\mathbf{C}) = \min_{\substack{\tau :=\|\boldsymbol{\upmu}\|_1+\|\boldsymbol{\upnu}\|_1  \\ (\boldsymbol{\upmu}, \boldsymbol{\upnu}) \in \mathcal{A}_\mathbf{C}}} \min_{\mathbf{P} \in \mathbf{U}( \mathbf{a}-\boldsymbol{\upmu}, \mathbf{b}-\boldsymbol{\upnu}) \cap M_\mathbf{C}(\mathbb{R}_+)}  \langle\mathbf{P}, \mathbf{C}\rangle_F,
\end{align}
in which the admissible set $\mathcal{A}_{\mathbf{C}}$ for the marginal blocked distributions $(\boldsymbol{\upmu}, \boldsymbol{\upnu})$ is defined as:
$$
\mathcal{A}_{\mathbf{C}}=\{(\boldsymbol{\upmu}, \boldsymbol{\upnu}) \in[0, \mathbf{a}] \times[0, \mathbf{b}] : \exists \mathbf{P} \in \mathbf{U}(\mathbf{a}-\boldsymbol{\upmu}, \mathbf{b}-\boldsymbol{\upnu}) \text { such that }\langle\mathbf{P}, \mathbf{C}\rangle<\infty\}.
$$
Here the inclusion $(\boldsymbol\upmu, \boldsymbol\upnu) \in[0, \mathbf{a}] \times[0, \mathbf{b}]$ is entrywise, namely, $\mu_i \in\left[0, a_i\right], i=$ $1,2, \ldots, n$, and $\nu_j \in\left[0, b_j\right], j=1,2, \ldots, m$. It is important to note that this bilevel optimization problem presents significant computational challenges. It was demonstrated that this bilevel minimization (\ref{eqn:sOT02}) can be transformed into an equivalent $l^1$-penalized optimal transport problem (\ref{eqn:sOT03}) as follows \cite{cang2022supervised}, which allows for efficient computation using the Dykstra algorithm when entropic regularization is applied:
\begin{definition}
The supervised optimal transport problem in bilevel form (\ref{eqn:sOT01}) or (\ref{eqn:sOT02}) is equivalent to the following single level optimization 
\begin{align}
    L_{\emph{sOT}}(\mathbf{a},\mathbf{b};\mathbf{C}) &= \min_{\substack{ \left(\boldsymbol{\upmu}, \boldsymbol{\upnu} \right) \in \mathcal{A}_{\mathbf{C}} \\ \mathbf{P} \in \mathbf{U}(\mathbf{a}-\boldsymbol{\upmu}, \mathbf{b}-\boldsymbol{\upnu}) \cap M_{\mathbf{C}}^+}} \langle\mathbf{P}, \mathbf{C}\rangle_F+ \gamma \Big(\| \boldsymbol{\upmu} \|_1 + \| \boldsymbol{\upnu} \|_1\Big) \quad \emph{for} \ \gamma\gg1 \nonumber\\
    &= \min_{\mathbf{P} \in \mathbf{U}(\leq \mathbf{a}, \leq \mathbf{b}) \cap M_{\mathbf{C}}^+} \langle\mathbf{P}, \mathbf{C}\rangle_F+ \gamma \Big(\| \mathbf{a}-\mathbf{P}\mathbf{1}  \|_1 + \| \mathbf{b}-\mathbf{P}^{\mathrm{T}} \mathbf{1} \|_1\Big) \quad \emph{for} \ \gamma\gg1. \label{eqn:sOT03}
\end{align}
\end{definition}

By incorporating entropic regularization into the formulation (\ref{eqn:sOT03}), the following expression is obtained:
\begin{align}
    L^{\epsilon}_{\text{sOT}}(\mathbf{a},\mathbf{b};\mathbf{C}) 
    & = \min_{\mathbf{P} \in \mathbf{U}(\leq \mathbf{a}, \leq \mathbf{b})  \cap M_{\mathbf{C}}^+} \langle\mathbf{P}, \mathbf{C}\rangle_F - \epsilon H(\mathbf{P}) + \gamma \Big(\| \mathbf{a}-\mathbf{P}\mathbf{1}  \|_1 + \| \mathbf{b}-\mathbf{P}^{\mathrm{T}} \mathbf{1} \|_1\Big) \\
    &= \min_{\mathbf{P} \in \mathbf{U}(\leq \mathbf{a}, \leq \mathbf{b})  \cap M_{\mathbf{C}}^+}  \epsilon \text{KL}(\mathbf{P}|\mathbf{K}) + \gamma \Big(\| \mathbf{a}-\mathbf{P}\mathbf{1}  \|_1 + \| \mathbf{b}-\mathbf{P}^{\mathrm{T}} \mathbf{1} \|_1\Big). \label{eqn:sOT_04}
\end{align}
Here, $H(\mathbf{P}) := -\sum_{ij} P_{ij}(\log P_{ij} -1 ) $ is the entropy function and $\mathbf{K} := \exp(-\mathbf{C}/\epsilon)$ is the kernel function. This formulation allows for the efficient application of the sOT solver first introduced in \cite{cang2022supervised}.

It is worth pointing out that for a given $\infty$-pattern in $\mathbf{C}$, the most possible transported mass by sOT is independent of the chosen cost form provided that $\gamma$ is sufficiently large. In other words, we can employ an $l^1$-cost, $l^2$-cost, or an all-one cost in $\mathbf{C}$, but assign the same $\infty$-pattern to $\mathbf{C}$. Then by choosing a sufficiently large $\gamma$, the resulting optimal plans $\mathbf{P}^*_{l^1}, \mathbf{P}^*_{l^2}, \mathbf{P}^*_{\text{all-one}}$ will transport an identical amount of mass. This property is later used in (\ref{eqn:sOT_computemass}) for the calculation of the maximum transported mass associated with a given 0-pattern in $\mathbf{P}$.

In this paper, we will introduce a Mirror-descent type iterative method to solve the supervised Gromov-Wasserstein problem. In each iteration, the problem can be transformed into the form that fits into (\ref{eqn:sOT_04}), utilizing the efficient sOT solver. More details can be found in Section \ref{subsection:Algorithm}.

\subsection{Gromov-Wasserstein Distance and Its Variants}

Conventional OT formulation and its variants focus on distributions $\mathcal{X}=\left\{\mathbf{x}_i\right\}_{i=1}^n$ and $\mathcal{Y}=\left\{\mathbf{y}_j\right\}_{j=1}^m$ in the same metric space with well-defined inter-distribution costs. On the other hand, we also often encounter distributions defined on different underlying spaces where such a cost is unavailable. To address this challenge, one solution proposed in \cite{memoli2007use,memoli2011gromov} extends the concept of OT to align two metric spaces by replacing the cost matrix with a quadratic function. In cases where the two discrete distributions originate from distinct underlying spaces and are unbalanced, a partial GW problem was introduced \cite{chapel2020partial} that considers the same admissible set as in the POT formulation. Remarkably, the GW problem can be reformulated as an order-4 tensor optimization problem. Building upon this approach, a generalized tensorized formulation was developed for unbalanced GW by incorporating two divergence terms into the objective function \cite{sejourne2023unbalanced}. 

In this section, we provide a brief overview of GW, partial GW (PGW) \cite{chapel2020partial}, and unbalanced GW (UGW) \cite{sejourne2021unbalanced} in the discrete setting and their equivalent formulations. To begin with, we will employ the same notations as used in Section \ref{subsection:OT}. Let $\mathcal{X}=\left\{\mathbf{x}_i\right\}_{i=1}^n$ and $\mathcal{Y}=\left\{\mathbf{y}_j\right\}_{j=1}^m$ be two point clouds that represent the source and target points, respectively, except that they are now taken from distinct metric spaces described by distance matrices $\mathbf{D}^1\in\R^{n\times n}$ and $\mathbf{D}^2\in\R^{m\times m}$. We continue to work with a pair of probability distributions $(\mathbf{a}, \mathbf{b}) \in \Sigma_n \times \Sigma_m$ over $\mathcal{X}$ and $\mathcal{Y}$ as in (\ref{eqn:point_clouds}). 
Then, the GW problem reads
\begin{align}\label{eqn:GW}
\text{GW}^{2}_{2}\left((\mathbf{a}, \mathbf{D}^{1}),\left(\mathbf{b}, \mathbf{D}^{2}\right)\right) \stackrel{\text { def. }}{=} \min _{\mathbf{P} \in \mathbf{U}(\mathbf{a}, \mathbf{b})} \sum_{i, j, k, l} \frac{1}{2}\left|{D}^{1}_{i k}-{D}_{j l}^{2}\right|^{2} {P}_{i j} {P}_{k l}.
\end{align}
Denoting a tensor-matrix multiplication as
\[
(\mathcal{M} \circ \mathbf{P})_{i j}=\sum_{k=1}^n \sum_{l=1}^m \mathcal{M}_{i j k l} {P}_{k l}, \quad \text{for\ } \mathcal{M}=\left(\mathcal{M}_{i j k l}\right) \in \mathbb{R}^{n \times m \times n \times m},\ \mathbf{P}\in\mathbb{R}^{n\times m},
\]
the objective function in GW (\ref{eqn:GW}) can be rewritten in the tensor form  
$$
    \sum_{i, j, k, l}\left|{D}^{1}_{i k}-{D}_{j l}^{2}\right|^2 {P}_{i j} {P}_{k l} = \left\langle\mathcal{M}\left( \mathbf{D}^{1},  \mathbf{D}^{2}\right) \circ \mathbf{P}, \mathbf{P}\right\rangle_F = \langle \mathcal{M}\left( \mathbf{D}^{1}, \mathbf{D}^2 \right), \mathbf{P} \otimes \mathbf{P} \rangle_F,
$$
where the cost tensor is given as
\begin{align}\label{eqn:4thorder_tensor}
\mathcal{M}\left( \mathbf{D}^{1}, \mathbf{D}^2 \right) _{ijkl }= \left |{D}_{i k}^1-{D}^{2}_{j l}\right |^2,
\end{align}
and the coupling tensor $\mathbf{P} \otimes \mathbf{P}$: $(\mathbf{P}\otimes\mathbf{P})_{ijkl} = P_{ij}P_{kl}$ is the outer product of $\mathbf{P}$ and itself. By expanding the quadratic term above, one can have an alternative GW formulation \cite{peyre2016gromov}:
\begin{align}
\text{GW}^{2}_{2}\left((\mathbf{a}, \mathbf{D}^{1}),\left(\mathbf{b}, \mathbf{D}^{2}\right)\right) 
& = \min _{\mathbf{P} \in \mathbf{U}(\mathbf{a}, \mathbf{b})} \frac{1}{2} \|\boldsymbol{\mathbf{P} \textbf{1}}\|^{2}_{\mathbf{D}^{1}} + \frac{1}{2}\|\boldsymbol{\mathbf{P}^{\mathrm{T}} \textbf{1}}\|^{2}_{\mathbf{D}^{2}} -  \left\langle \mathbf{D}^{1} \mathbf{P} \mathbf{D}^{2,\mathrm{T}}, \mathbf{P} \right\rangle_F \nonumber\\
& = \frac{1}{2} \|\mathbf{a}\|^{2}_{\mathbf{D}^{1}\odot\mathbf{D}^1} + \frac{1}{2} \|\mathbf{b}\|^{2}_{\mathbf{D}^{2}\odot\mathbf{D}^2} -  \max _{\mathbf{P} \in \mathbf{U}(\mathbf{a}, \mathbf{b})}  \left\langle \mathbf{D}^{1} \mathbf{P} \mathbf{D}^{2,\mathrm{T}}, \mathbf{P} \right\rangle_F, \label{eqn:GW02}
\end{align}
where $\|\mathbf{a}\|^2_{\mathbf{D}^1\odot\mathbf{D}^1}: = \mathbf{a}^{\mathrm{T}}(\mathbf{D}^1\odot\mathbf{D}^1)\mathbf{a}$, and similarly for $\|\mathbf{b}\|^2_{\mathbf{D}^2\odot\mathbf{D}^2}$.

When $\|\mathbf{a}\|_1 \neq \|\mathbf{b}\|_1$, a partial GW problem (PGW) can be formulated similar to POT, which reads \cite{chapel2020partial}
\begin{align}
\text{PGW}^{2}_{2}\left((\mathbf{a}, \mathbf{D}^{1}),\left(\mathbf{b}, \mathbf{D}^{2}\right);s\right) \stackrel{\text { def. }}{=}& \min _{\mathbf{P} \in \mathbf{U}^{s}(\leq \mathbf{a}, \leq \mathbf{b})} \frac{1}{2}\sum_{i, j, k, l}\left|{D}^{1}_{i k}-{D}_{j l}^{2}\right|^2 {P}_{i j} {P}_{k l} \nonumber \\
=& \min _{\mathbf{P} \in \mathbf{U}^{s}(\leq \mathbf{a}, \leq \mathbf{b})} \frac{1}{2} \langle \mathcal{M}\left( \mathbf{D}^{1}, \mathbf{D}^2 \right), \mathbf{P} \otimes \mathbf{P} \rangle_F. \label{eqn:PGW}
\end{align} 

\noindent It is evident that GW (\ref{eqn:GW}) and the PGW (\ref{eqn:PGW}) are order-4 tensor optimization problem. However, a notable distinction between them arises. Since PGW has no equality constraints $\mathbf{P} \mathbf{1} = \mathbf{a}, \mathbf{P}^{\mathrm{T}} \mathbf{1} = \mathbf{b}$, consequently there is no alternative maximization form as in the GW form (\ref{eqn:GW02}). Besides, separately rescaling the distance matrices may lead to different minimizers for PGW.

Similar to unbalanced OT, unbalanced GW (UGW) has been proposed as follows \cite{sejourne2021unbalanced}:
\begin{align}
\mathrm{UGW}((\mathbf{a}, \mathbf{D}^{1}),\left(\mathbf{b}, \mathbf{D}^{2}\right))=\min _{\mathbf{P} \in \mathbf{U}(\leq \mathbf{a}, \leq \mathbf{b})} & \frac{1}{2} \langle \mathcal{M}\left( \mathbf{D}^{1}, \mathbf{D}^2 \right), \mathbf{P} \otimes \mathbf{P} \rangle_F \nonumber \\
&+ \mathrm{D}_{\varphi}(\mathbf{P} \textbf{1} \otimes \mathbf{P} \textbf{1} \mid \mathbf{a} \otimes \mathbf{a}) + \mathrm{D}_{\varphi}(\mathbf{P}^{\mathrm{T}} \textbf{1} \otimes \mathbf{P}^{\mathrm{T}} \textbf{1} \mid \mathbf{b} \otimes \mathbf{b}), \label{eqn:UGW}
\end{align}
where $D_{\varphi}(\cdot)$ is the quadratized $\varphi$-divergence, and particular choices of $\varphi$ are KL divergence:  $\varphi(t) = t \log (t) - t + 1 $ and total variation:  $\varphi(t) = |t - 1|$.



\section{Supervised Gromov-Wasserstein Optimal Transport}\label{sec:main}

PGW and UGW effectively tackle the issue of unequal mass from the two marginal distributions from a global perspective, that is setting a total transported mass in PGW and using two global divergence penalties in UGW. It is also common in practice that the preservation of geometry needs to be controlled which leads to a collection of local constraints. For example, consider the comparison of two pairs, ${\mathbf{x}_i, \mathbf{x}_j}$ and ${\mathbf{y}_k, \mathbf{y}_l}$ originating from two distinct spaces. If the distance between $\|\mathbf{x}_i - \mathbf{x}_j\|$ and $\|\mathbf{y}_k - \mathbf{y}_l\|$ exceeds a certain threshold, we may need to disregard their alignment. It is motivated by such scenarios that we introduce the supervised Gromov-Wasserstein (sGW) optimal transport problem.

\subsection{sGW Formulation}

A key feature in the sGW optimal transport formulation is the presence of $\infty$ elements in the cost $\mathcal{M}\left( \mathbf{D}^{1}, \mathbf{D}^2 \right)$. Inspired by the sOT formulation \cite{cang2022supervised}, we aim to find the best possible matching between datasets from two different metric spaces, involving as many points as possible (transporting as much mass as possible from the marginal distributions) while preserving the geometric similarity between them as much as possible. Motivated by this idea, we formulate the sGW as follows:
\begin{definition}
The supervised Gromov-Wasserstein optimal transport problem is defined as the following bilevel optimization:
\begin{align}\label{eqn:sGW01}
\emph{sGW}((\mathbf{a}, \mathbf{D}^{1}),\left(\mathbf{b}, \mathbf{D}^{2}\right)) = \max_{s \in [0, \min \{ \|\mathbf{a}\|_{1}, \|\mathbf{b}\|_{1} \} ]} \min_{\mathbf{P} \in \mathbf{U}(\leq \mathbf{a}, \leq \mathbf{b}) } \Big\{ \frac{1}{2} \langle \mathcal{M}, \mathbf{P} \otimes \mathbf{P} \rangle_F: \langle \mathbf{P} \otimes \mathbf{P} , \mathbf{1} \rangle_F= s^2 \Big\}.
\end{align}
Here the tensor $\mathcal{M} = \mathcal{M}(\mathbf{D}^1, \mathbf{D}^2)$ is defined in (\ref{eqn:4thorder_tensor}), and contains some $\infty$ elements.
\end{definition}

The $\infty$-pattern of the tensor $\mathcal{M}$ can be general and here we focus on a natural choice in applications:
\begin{align}\label{eqn:Cost_cut}
\mathcal{M}\left( \mathbf{D}^{1}, \mathbf{D}^2 \right) _{ijkl }= 
\begin{cases}
 |{D}_{i k}^1-{D}^{2}_{j l} |^2, \quad \text{if} \  |{D}_{i k}^1-{D}^{2}_{j l}| \le \rho, \\
\infty, \hspace{0.75in} \text{if} \ |{D}_{i k}^1-{D}^{2}_{j l}| > \rho.
\end{cases}
\end{align}
where $\rho \in ( 0 , \max _{ijkl} \sqrt{\mathcal{M}_{ijkl}} )$ is a fixed threshold. From now on, we will focus on the uniform marginal distributions:
\begin{align}\label{eqn:uniform_marginal}
\mathbf{a}=\sum_{i=1}^n \frac{1}{n} \delta_{\boldsymbol{x}_i}, \quad \mathbf{b}=\sum_{j=1}^m \frac{1}{m} \delta_{\boldsymbol{y}_j},
\end{align}
as they are of utmost importance in GW applications for matching objects.

\subsection{Order Reduction of Feasible Set for sGW using Minimal Vertex Cover}

The entire admissible set directly induced by the constraints on the cost tensor (\ref{eqn:Cost_cut}) is impractical to consider in real applications due to the high order. To mitigate this limitation, we introduce a connection between sGW and minimal vertex cover (MVC), where MVC will be used to narrow down the original admissible set induced by an order-4 constraint toward a much smaller set defined as a collection of order-2 constraints on the transport plan.

Given a fixed $\rho \in ( 0 , \max _{ijkl} \sqrt{\mathcal{M}_{ijkl}} )$, we define the set $\mathcal{S}_\mathcal{M}$ as follows:
\begin{align}
\mathcal{S}_\mathcal{M} \stackrel{\emph { def. }}{=}\left\{\mathbf{P} \in \mathbb{R}_{+}^{n \times m}: {P}_{ij} \cdot {P}_{kl} = 0 \quad \text{if} \  \mathcal{M}_{ijkl} = \infty, \quad \forall 1 \leq i,k \leq n, \ 1 \leq j,l \leq m \right\}.
\end{align}
It is evident that sGW (\ref{eqn:sGW01}) can be reformulated as
\begin{align}\label{eqn:sGW02}
\text{sGW} = \max_{s \in [0, \min \{ \|\mathbf{a}\|_{1}, \|\mathbf{b}\|_{1} \} ]} \min_{\mathbf{P} \in \mathbf{U}(\leq \mathbf{a}, \leq \mathbf{b}) \cap \mathcal{S_M} } \Big\{ \frac{1}{2} \langle \mathcal{M}, \mathbf{P} \otimes \mathbf{P} \rangle_F: \langle \mathbf{P} \otimes \mathbf{P} , \mathbf{1} \rangle_F= s^2 \Big\}.
\end{align}
In other words, the $\infty$-pattern in $\mathcal{M}$ is equivalently replaced by a $0$-pattern in $\mathbf{P}$ through the set $\mathcal{S}_{\mathcal{M}}$. The $0$-pattern of $\mathbf{P}$ in $\mathcal{S}_{\mathcal{M}}$ is defined via a product ${P}_{ij} \cdot {P}_{kl} = 0$, which can be challenging to implement in practice. We therefore look for an alternative set from which one can simply indicate which $P_{ij}$ is zero. To this end, we define a graph and consider the Minimal Vertex Covers (MVCs) on it.

Let $V = \{v_{ij}\}_{1\le i \le n, 1\le j \le m}$ be a set of $n\times m$ vertices. We define an edge set $E$ by the following rule: there is an edge connecting $v_{ij}$ and $v_{kl}$ if and only if ${P}_{ij} \cdot {P}_{kl} = 0$. We then define a graph by the vertice set $V$ and edge set $E$:
\begin{align}\label{eqn:graphG}
    G = G(V,E).
\end{align}
Suppose $V_{\text{min}}^{(l)}, l = 1, \cdots, N_{\text{mvc}}$, represents the list of all possible MVC of $G(V, E)$ (see Section \ref{subsubsection:GreddyMVC} for the definition of MVC), and $\mathcal{T}_\mathcal{M}^{(l)}$ represents 
\begin{align}\label{eqn:T_M}
\mathcal{T}_\mathcal{M}^{(l)} \stackrel{\emph { def. }}{=}\left\{\mathbf{P} \in \mathbb{R}_{+}^{n \times m}: P_{ij}= 0 \quad \text{if} \ v_{ij} \in V_{\text{min}}^{(l)}, \quad \forall 1 \leq i \leq n, \ 1 \leq j \leq m \right\}.
\end{align}
By the definition of MVC, we have that $\mathcal{T}_{\mathcal{M}}^{(l)} \subset \mathcal{S}_{\mathcal{M}}$ for each $l$, and $\cup_l\mathcal{T}_{\mathcal{M}}^{(l)} = \mathcal{S}_{\mathcal{M}}$. Therefore instead of solving sGW over $\mathcal{S}_{\mathcal{M}}$, we can recast (\ref{eqn:sGW02}) into 
\begin{align}\label{eqn:sGW03}
\text{sGW} = \max_{s \in [0, \min \{ \|\mathbf{a}\|_{1}, \|\mathbf{b}\|_{1} \} ]} \min_{\mathbf{P} \in \mathbf{U}(\leq \mathbf{a}, \leq \mathbf{b}) \cap \left(\cup_l \mathcal{T}_\mathcal{M}^{(l)}\right) } \Big\{ \frac{1}{2} \langle \mathcal{M}, \mathbf{P} \otimes \mathbf{P} \rangle_F: \langle \mathbf{P} \otimes \mathbf{P} , \mathbf{1} \rangle_F= s^2 \Big\}.
\end{align}
Now the formulation (\ref{eqn:sGW03}) is to minimize the same objective function over the union of smaller sets $\mathcal{T}_{\mathcal{M}}^{(l)}$, which directly characterizes the $0$-pattern for each possible $P_{ij}$ (order-2 constraints), instead of featuring it by $P_{ij}\cdot P_{kl}$ (a order-4 constraint).

Given a MVC $V_{\text{min}}^{(l)}$ and the corresponding $\mathcal{T}_{\mathcal{M}}^{(l)}$, one can find the most possible transported mass $s^{(l)}$ by simply solving the following sOT problem:
\begin{align}\label{eqn:sOT_computemass}
    L_{\text{sOT}}(\mathbf{a},\mathbf{b};\mathbf{C}) 
    = \min_{\mathbf{P} \in \mathbf{U}(\leq \mathbf{a}, \leq \mathbf{b}) \cap \mathcal{T}_{\mathcal{M}}^{(l)}} \langle\mathbf{P}, \mathbf{C}\rangle_F+ \gamma \Big(\| \mathbf{a}-\mathbf{P}\mathbf{1}  \|_1 + \| \mathbf{b}-\mathbf{P}^{\mathrm{T}} \mathbf{1} \|_1\Big) \quad \emph{for} \ \gamma\gg1.
\end{align}
Here we take $\mathbf{C}$ to be the all-one matrix. Once a minimizer $\mathbf{P}^*$ is found by using the sOT solver introduced in \cite{cang2022supervised}, we take $s^{(l)} = \langle \mathbf{P}^*, \mathbf{1} \rangle$. 

While only considering MVCs will significantly reduce the original admissible set, there could still be numerous MVCs for a relatively large graph $G(V,E)$, and obtaining all MVCs is an NP-hard problem. However, using various experiments, we demonstrate that a small set of MVCs is likely sufficient to obtain nearly optimal solutions in practice. We numerically verified that the most possible transported mass $s$ by a $V_{\text{min}}$ statistically decreases with respect to the size of $V_{\text{min}}$, as illustrated in Fig. \ref{fig:StatisticalTrend}. In the figure, we analyze point clouds of two disks with radius 1 and 2, respectively. Fixing the threshold $\rho = 0.2$, we apply GVC algorithm described in Section \ref{subsubsection:GreddyMVC} to generate MVCs and then use the sOT solver from Section \ref{Section:sOT_solver} to compute the corresponding transported mass $s$. Before each simulation by the GVC algorithm, we randomly pick a percentage $p_{\text{mvc}}$ between 0 and 50\%, and then randomly fix $p_{\text{mvc}}$-percent vertices from the graph $G(V,E)$ in (\ref{eqn:graphG}). Subsequently, we employ the GVC algorithm on the remaining $(1-p_{\text{mvc}})$-percent vertices to generate a vertex cover (VC). This VC combined with the fixed $p_{\text{mvc}}$-percent vertices are then trimmed to obtain an MVC. This process, depicted in the middle subfigure, results in 100 MVCs, with the corresponding transported mass. An evident statistical decrease in the transported mass $s$ is observed as the size of the MVCs increases. In the right subfigure, we examine this trend using various values of $\rho$, each point generated by using the same method as in the middle subfigure. This result still reveals a similar decline in the transported mass with increasing MVC size. 

\begin{figure}[t] 
\begin{center}
\includegraphics[width=0.3\linewidth]{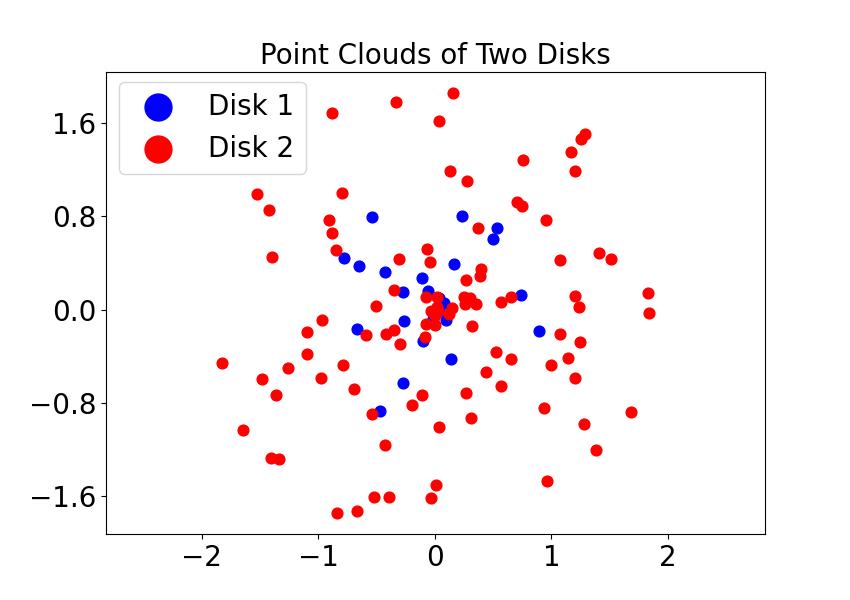}
\includegraphics[width=0.3\linewidth]{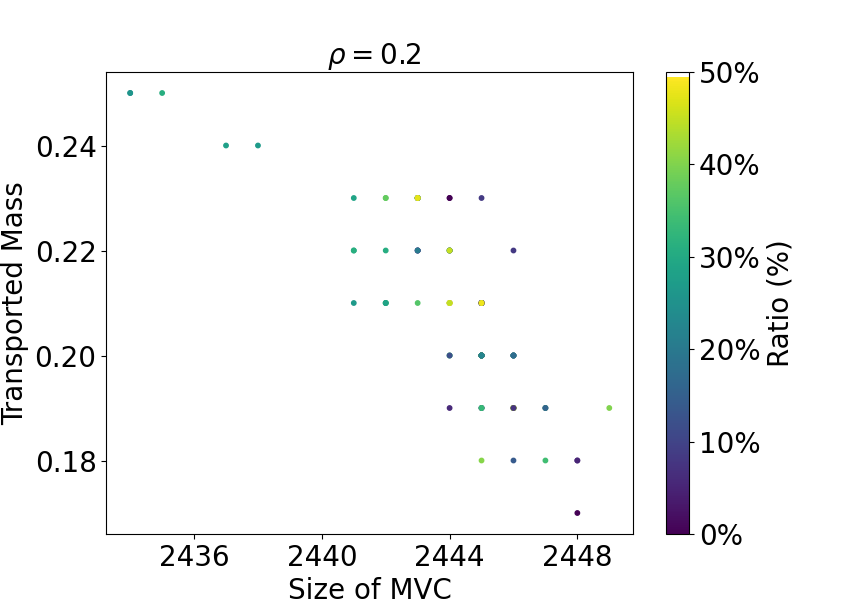}
\includegraphics[width=0.3\linewidth]{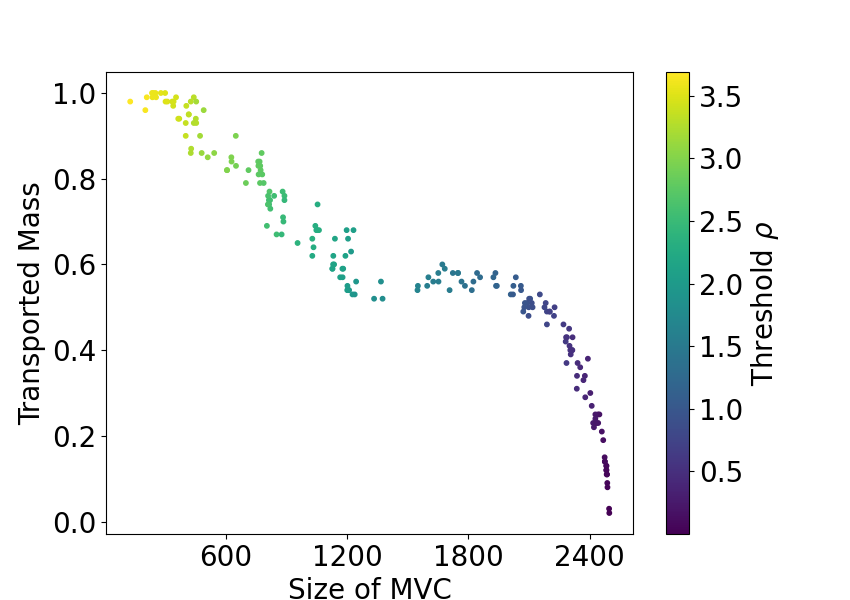}
\end{center}
\caption{Statistically decreasing trend of the transported mass $s$ with respect to the size of MVC. Left: two point clouds generated over two disks, one of radius 1 (blue), the other of radius 2 (red). Middle: for fixed $\rho = 0.2$, 100 MVCs are generated and the corresponding transported masses are calculated by sOT solver. A clear decreasing trend is observed. Right: for various values of $\rho$, MVCs and the corresponding masses are calculated, from which a similar decreasing trend is presented.}
\label{fig:StatisticalTrend}
\end{figure}

Therefore, in practice, we can select a small subset of all MVCs, say hundreds of them, compute the transported mass for each selected MVC, and then choose the one with the greatest transported mass (and also statistically the one with the smallest size), denoted by $\mathcal{T}_{\mathcal{M}}^*$, with transported mass $s^*$. Then the sGW (\ref{eqn:sGW03}) can be approximately reformulated as
\begin{align}\label{eqn:sGW04}
    \text{sGW} = \max_{s \in [0, \min \{ \|\mathbf{a}\|_{1}, \|\mathbf{b}\|_{1} \} ]} \min_{\mathbf{P} \in \mathbf{U}(\leq \mathbf{a}, \leq \mathbf{b}) \cap \mathcal{T}_\mathcal{M}^{*} } \Big\{ \frac{1}{2} \langle \mathcal{M}, \mathbf{P} \otimes \mathbf{P} \rangle_F: \langle \mathbf{P} \otimes \mathbf{P} , \mathbf{1} \rangle_F= s^2 \Big\}.
\end{align}

\subsection{sGW in Penalty Form}

Note that the approximate sGW formulation (\ref{eqn:sGW04}) is a bilevel optimization, which is still challenging for numerical implementation. For the sake of easy implementation, we introduce an $l^1$-penalized reformulation to further approximate (\ref{eqn:sGW04}) which is given below.
\begin{definition}
The approximate supervised Gromov-Wasserstein optimal transport problem (\ref{eqn:sGW04}) can be recast into the following $l^1$-penalized 4th-order tensor optimization 
\begin{align}\label{eqn:sGW05}
\mathrm{sGW} = 
 \min _{\mathbf{P} \in \mathbf{U}(\leq \mathbf{a}, \leq \mathbf{b}) \cap \mathcal{T}_\mathcal{M}^{*}  } \frac{1}{2} \langle \mathcal{M}, \mathbf{P} \otimes \mathbf{P} \rangle_F + \gamma \left( \| \mathbf{a}-\mathbf{P} \mathbf{1}  \|_{1} + \| \mathbf{b}-\mathbf{P}^{\mathrm{T}} \mathbf{1} \|_{1} \right), \text{ for } \gamma \gg 1,
\end{align}
where the tensor cost $\mathcal{M} = \mathcal{M}\left( \mathbf{D}^{1}, \mathbf{D}^2 \right)$ defined in Eq. (\ref{eqn:4thorder_tensor}) contains $\infty$ elements. 
\end{definition}

It's worth pointing out that the $l^1$-penalized $\gamma$-term is due to the constraint that we need to transport the maximum possible marginal mass from $\mathbf{a}$ to $\mathbf{b}$. A sufficiently large penalty constant $\gamma$ is taken to enforce this constraint.

\subsection{Properties of sGW and Connection to Other Variants}

In Section \ref{subsection:sOT}, we discussed the relationship between sOT and POT. Essentially, sOT addresses the POT problem with the maximum possible total mass $s=\min\{\|\mathbf{a}\|_1,\|\mathbf{b}\|_1\}$, when the cost matrix $\mathbf{C}$ contains no $\infty$ entries. The sGW shares a similar relationship with PGW. In fact, sGW (\ref{eqn:sGW01}) degenerates to the PGW with $s=\min \{\|\boldsymbol{a}\|_1,\|\boldsymbol{b}\|_1\}$ provided that the tensor cost $\mathcal{M}$ contains no $\infty$ entries.

The equivalence between the bilevel formulation (\ref{eqn:sOT01}) of sOT and the $l^1$-penalized formulation (\ref{eqn:sOT03}) has been rigorously proven in \cite{cang2022supervised}. However, it remains unclear whether the bilevel formulation (\ref{eqn:sGW01}) of sGW is equivalent to the $l^1$-penalized formulation (\ref{eqn:sGW02}). Here we take $l^1$-penalized formulation (\ref{eqn:sGW02}) as an approximation to the bilevel formulation (\ref{eqn:sGW01}).

To numerically solve the sGW problem presented in \ref{eqn:sGW05}, we first seek a proper MVC from the graph $G(V,E)$ which can be computationally costly. Consequently, our approach involves downsampling the point clouds and solving MVC and the corresponding sGW (\ref{eqn:sGW05}) on two smaller datasets, followed by upsampling the results to the original datasets. Further details can be found in Section \ref{subsection:Algorithm}.

The traditional iterative methods for solving GW, PGW, and UGW may not be suitable to sGW due to the presence of an $\infty$-pattern in the tensor cost $\mathcal{M}$. On the other hand, the sOT solver in \cite{cang2022supervised} can not be directly applied to sGW (\ref{eqn:sGW05}). This is because, if one directly transfers (\ref{eqn:sGW05}) into an iterative form with a pre-calculated $\mathcal{T}_{\mathcal{M}}^*$:
\begin{align}\label{eqn:sGW_Iteration01}
\mathbf{P}^{(k+1)} =
 \argmin _{\mathbf{P} \in \mathbf{U}(\leq \mathbf{a}, \leq \mathbf{b}) \cap  \mathcal{T}_\mathcal{M}^*   }  \frac{1}{2} \langle \mathcal{M}\circ \mathbf{P}^{(k)}, \mathbf{P} \rangle_F + \gamma \left( \| \mathbf{a}-\mathbf{P} \mathbf{1}  \|_{1} + \| \mathbf{b}-\mathbf{P}^{\mathrm{T}} \mathbf{1} \|_{1} \right),
\end{align}
the $\infty$-pattern of $\mathcal{M}$ at $k$-th iteration will change due to the influence of $\mathbf{P}^{(k)}$. To address this issue, we will replace the $\infty$-pattern equivalently by the $0$-pattern of $\mathbf{P}$ as defined by $\mathcal{T}_{\mathcal{M}}^*$, and fix the $0$-pattern in $\mathbf{P}$ at each iteration. When applying the sOT solver to the iteration (\ref{eqn:sGW_Iteration01}), we will assign an $\infty$-pattern to $\mathcal{M}\circ \mathbf{P}^{(k)}$ according to the $0$-pattern of $\mathbf{P}$.

The threshold parameter $\rho$ quantitatively controls the geometric similarity between two datasets and introduces a diffusion effect on the matching as will be explained in Section \ref{subsection:Algorithm}. Following the sOT heuristics, the $l^1$ penalty in sGW respects the polytope structure of the admissible set and discourages the undesirable splitting matchings, especially when $\mathbf{a}$ and $\mathbf{b}$ are two uniform distributions.

\subsection{Numerical Algorithms for sGW}\label{subsection:Algorithm}

To numerically solve the sGW problem, we incorporate the conventional entropy regularization into (\ref{eqn:sGW04}) and define the entropy regularized sGW as:
\begin{align}\label{eqn:sGW06}
    \text{sGW}^{\epsilon} =
    \min _{\mathbf{P} \in \mathbf{U}(\leq \mathbf{a}, \leq \mathbf{b}) \cap  \mathcal{T}_\mathcal{M}^*}  \frac{1}{2}\langle \mathcal{M}, \mathbf{P} \otimes \mathbf{P} \rangle_F - \epsilon H(\mathbf{P}) + \gamma \left( \| \mathbf{a}-\mathbf{P} \mathbf{1} \|_{1} + \| \mathbf{b}-\mathbf{P}^{\mathrm{T}} \mathbf{1} \|_{1} \right).
\end{align}
Now we aim to solve the entropy regularized sGW (\ref{eqn:sGW06}). In this subsection, we present the full algorithm for solving (\ref{eqn:sGW06}) and prove its convergence. There are mainly two steps for solving the entropy regularized sGW (\ref{eqn:sGW06}). Step I is to find a proper MVC $V_{\text{min}}^*$ from the graph $G(V,E)$ defined in (\ref{eqn:graphG}) and construct $\mathcal{T}_{\mathcal{M}}^*$. Step II is to iteratively solve for the optimal plan $\mathbf{P}^*$ given the $V_{\min}^*$.

\subsubsection{Greedy MVC Algorithm}\label{subsubsection:GreddyMVC}

In graph theory, finding MVCs \cite{erickson2019algorithms,cormen2022introduction} is a classic optimization problem, especially for finding the MVC of the smallest size, that is, the minimum vertex cover(s). Formally, given an undirected graph $G=(V, E)$ with $V$ representing the set of vertices and $E$ representing the set of edges, a vertex cover $VC \subseteq V$ is a set of vertices that includes at least one endpoint of every edge of the graph. In other words, for every edge $(u, v) \in E$, either $u$ is in $VC$ or $v$ is in $VC$ or both. a minimal vertex cover (MVC) is defined as a vertex cover that does not contain another vertex cover. 
Finding all MVCs and the minimum vertex covers is known to be NP-hard, making it computationally intractable for large graphs. Unless the graph has a bipartite or tree structure, there is currently no known polynomial-time algorithm capable of solving all instances of this problem.

A classical algorithm for the MVC problem is to obtain a vertex cover with size at most twice the optimal one. The algorithm selects an edge from the graph $G$, includes both of its endpoints into the vertex cover, and then eliminates these two vertices along with all their adjacent edges from $G$. This procedure is iteratively executed until there are no remaining edges in $G$ (see \cite{cormen2022introduction} section 35.1). However, this approach always introduces numerous extra unrelated vertices $P_{ij}$ and contradicts the geometric significance of the matching. Given a sufficiently large value for $\gamma$ in the sGW, we prefer to construct an MVC by collecting the vertex of the highest degree and deleting all its related edges and vertices in each iteration. This is the so-called Greedy Vertex Cover (GVC) algorithm \cite{erickson2019algorithms} (see Algorithm \ref{alg:gvc}). In addition, we introduce some randomness in selecting the vertex of the highest degree in each iteration, as there may be multiple vertices of the highest degree. Note that the GVC algorithm only generates an approximate, rather than exact, MVC. However, we numerically confirm that a large part of the outcomes generated by GVC algorithm are MVC. For example, 90\% of the VCs generated by GVC are MVCs in the point clouds examples in Figure \ref{fig:cross_example}. 

It is important to stress that even after we use GVC algorithm to generate $N$ outcomes, and the one transported the most mass $s^*$ is not a MVC, denoted by $V^*$, we can still use $V^*$ to define $\mathcal{T}_{\mathcal{M}}^*$ through the definition (\ref{eqn:T_M}) and then solve the entropic sGW (\ref{eqn:sGW06}). These approximations are effective for our goal since we seek VCs with small sizes that likely lead to large transported mass, rather than requiring exact MVCs.
By generating sufficiently many outcomes, which may or may not be MVC, from the GVC algorithm, we choose the one with the greatest transported mass. This completes the construction of $\mathcal{T}_\mathcal{M}^*$.

\begin{algorithm}

\caption{GVC algorithm for the construction of $\mathcal{T}_\mathcal{M}^*$}\label{alg:gvc}

Input: distance matrices $\textbf{D}^1$ and $\textbf{D}^2$; threshold parameter $\rho > 0$; 

Step I: Construct graph $G(V,E)$ from the 4th-order tensor $\mathcal{M}$ in (\ref{eqn:Cost_cut}):
    \begin{itemize}
        \item Set $V = \{v_{i,j}\}_{1\le i \le n, 1\le j \le m}$,
        \item Build an edge $(v_{ij}, v_{kl}) \in E $  if and only if $\mathcal{M}_{ijkl} > \rho$.\\
    \end{itemize}
    
Step II: Greedy minimum vertex cover of $G(V,E):$

\qquad for $l = 1, \cdots, N_{\text{mvc}}$

\qquad \qquad $V_{\min}^{(l)} \leftarrow \varnothing$, \quad $G_0 \leftarrow G(V,E)$, \quad $p \leftarrow 0$ 

\qquad \qquad  while $G_p$ has at least one edge 

\qquad\qquad\qquad $p \leftarrow p+1$ ;

\qquad\qquad\qquad $v_p \leftarrow $ choose a vertex in $G_{p-1}$ with maximum degree randomly;

\qquad\qquad\qquad  $d_p \leftarrow \operatorname{deg}_{G_{p-1}}\left(v_p\right)$ ;

\qquad\qquad\qquad  $G_p \leftarrow G_{p-1} \backslash \{v_p\}$ ;

\qquad\qquad\qquad  $V_{\min}^{(l)} \leftarrow V_{\min}^{(l)} \cup \{v_p\}$ ;

\qquad return $\{V_{\min}^{(l)}\}_{l=1}^{N_{\text{mvc}}}$. \\

Step III: Find the transported mass $s^{(l)}$ for each $V_{\min}^{(l)}$ by sOT solver in Algorithm \ref{alg:sot} by taking $\mathbf{K} = e^{-\mathbf{C}/\epsilon}$ with $\mathbf{C}$ being an all-one matrix with an $\infty$-pattern consistent with $V_{\min}^{(l)}$:

\qquad for $l = 1, \cdots, N_{\text{mvc}}$

\qquad \qquad find the optimal $\mathbf{P}^{(l)}$ by sOT solver in Algorithm \ref{alg:sot};

\qquad \qquad $s^{(l)} = \langle \mathbf{P}^{(l)}, \mathbf{1} \rangle_F$;

\qquad return $\{s^{(l)}\}_{l=1}^{N_{\text{mvc}}}$\\

Step IV: Select the largest transported mass $s^{(l^*)}$, set $V_{\text{min}}^* = V_{\min}^{(l^*)}$ and construct $\mathcal{T}_\mathcal{M}^*$ by (\ref{eqn:T_M}).

\end{algorithm}

\subsubsection{Supervised Optimal Transport Solver}\label{Section:sOT_solver}

After the construction of $\mathcal{T}_{\mathcal{M}}^*$ to assign $0$-pattern for $\mathbf{P}$, we adopt the KL Mirror-C descent algorithm \cite{beck2017first} for the non-convex objective function in (\ref{eqn:sGW06}). To this end, we rewrite (\ref{eqn:sGW06}) as
\begin{align} \label{eqn:noncovex_objective}
\min_{\mathbf{P} \in \mathbb{R}_{+}^{n \times m} } \mathcal{J}(\mathbf{P})+\mathcal{K}(\mathbf{P})
\end{align}
where
\begin{align*}
\mathcal{J}(\mathbf{P}) &= \frac{1}{2} \langle \mathcal{M}, \mathbf{P} \otimes \mathbf{P}\rangle_F,\\
\mathcal{K}(\mathbf{P}) &= -\epsilon H(\mathbf{P}) + \iota_{\mathbf{U}(\leq \mathbf{a}, \leq \mathbf{b}) \cap \mathcal{T}_\mathcal{M}^*} (\mathbf{P}) + \gamma \left( \| \mathbf{a}-\mathbf{P} \textbf{1}  \|_{1} + \| \mathbf{b}-\mathbf{P}^\mathrm{T} \textbf{1} \|_{1} \right).
\end{align*}
Here the term $\iota_{\mathbf{U}(\leq \mathbf{a}, \leq \mathbf{b}) \cap \mathcal{T}_\mathcal{M}^*}$ in $\mathcal{K}(\mathbf{P})$ is the indicator of the admissible set. More specifically, we have the following lemma when applying Mirror-C algorithm to entropic sGW (\ref{eqn:sGW06}).
\begin{lemma}
When applying the Mirror-C algorithm to the entropic sGW problem (\ref{eqn:sGW06}), we obtain the following iterative scheme:
\begin{align}\label{eqn:MirrorC_01}
    \mathbf{P}^{(k+1)} = \underset{\mathbf{P}\in \mathcal{T}_\mathcal{M}^*}{\mathrm{argmin}} \  \emph{KL}(\mathbf{P}|\mathbf{Q}^{(k)}) + h_1(\mathbf{P1}) + h_2(\mathbf{P}^{\top}\mathbf{1}),
\end{align}
where
\begin{align*}
    &\mathbf{Q}^{(k)} = \exp \left( -\frac{\Delta t}{1+\epsilon\Delta t} \mathcal{M}\circ \mathbf{P}^{(k)} + \frac{1}{1+\epsilon\Delta t} \log \mathbf{P}^{(k)} \right), \\
    &h_1(\mathbf{P1}) = \frac{\Delta t}{1+\epsilon\Delta t}\gamma \|\mathbf{a}-\mathbf{P1}\|_1 + \iota_{[0,\mathbf{a}]}(\mathbf{P1}),\\
    &h_2(\mathbf{P}^{\top}\mathbf{1}) = \frac{\Delta t}{1+\epsilon\Delta t}\gamma \|\mathbf{b}-\mathbf{P}^{\top}\mathbf{1}\|_1 + \iota_{[0,\mathbf{b}]}(\mathbf{P}^{\top}\mathbf{1}).
\end{align*}
\end{lemma}
\begin{proof}
    Applying the Mirror-C algorithm to sGW problem, we have
    \begin{align*}
        \mathbf{P}^{(k+1)} & = \underset{\mathbf{P}}{\mathrm{argmin}}\left\{ \left\langle \nabla \mathcal{J}(\mathbf{P}^{(k)}), \mathbf{P} \right\rangle + \mathcal{K}(\mathbf{P}) + \frac{1}{\Delta t} \text{KL}(\mathbf{P}|\mathbf{P}^{(k)}) \right\}.
    \end{align*}
    Noting that $\nabla \mathcal{J}(\mathbf{P}) = \mathcal{M}\circ \mathbf{P}$, it follows that
    \begin{align*}
        \mathbf{P}^{(k+1)} & = \underset{\mathbf{P}\in\mathcal{T}_{\mathcal{M}}^*}{\mathrm{argmin}}
        \Big\{ 
        \Delta t 
        \left\langle 
        \mathcal{M}\circ \mathbf{P}^{(k)}, \mathbf{P} 
        \right\rangle 
        -\epsilon \Delta t H(\mathbf{P}) + \iota_{\mathbf{U}(\le\mathbf{a},\le\mathbf{b})}(\mathbf{P}) \\
        & \hspace{1.2in} +  \gamma \Delta t \left( \| \mathbf{a}-\mathbf{P} \textbf{1}  \|_{1} + \| \mathbf{b}-\mathbf{P}^\mathrm{T} \textbf{1} \|_{1} \right)  + \big(-H(\mathbf{P}) - \langle \mathbf{P}, \log \mathbf{P}^{(k)} \rangle \big) 
        \Big\}.    
    \end{align*}
    Combining the two terms of $H(\mathbf{P})$ and noting that $-H(\mathbf{P}) - \langle \mathbf{P}, \mathbf{Q} \rangle = \text{KL}(\mathbf{P}|\exp(\mathbf{Q}))$ for a generic $\mathbf{Q}$, it yields the Mirror-C iteration (\ref{eqn:MirrorC_01}).
\end{proof}

The iteration (\ref{eqn:MirrorC_01}) aligns with the sOT solver \cite{cang2022supervised}. The only difference is that we specify the $0$-pattern of $\mathbf{P}$ rather than the $\infty$-pattern of $-\log \mathbf{Q}^{(k)}$. This can be easily rectified by assigning an identical $0$-pattern to $\mathbf{Q}^{(k)}$ from $\mathbf{P} \in \mathcal{T}_{\mathcal{M}}^*$. We present the sOT solver as follows for easy reference; for more detailed derivation, please refer to \cite{cang2022supervised}.

\begin{algorithm}

\caption{sOT solver in log-domain}\label{alg:sot}

Input: $\mathbf{f}^{(0)}=\mathbf{g}^{(0)}=\mathbf{0}$, $\mathbf{K} = \mathbf{Q}^{(k)}$; \\
General step: for $l=0, 1, 2, \cdots$, execute the following steps:
\begin{align}
    &\mathbf{f}^{(l+1)} = \min \left\{ \epsilon\log(\mathbf{a}) -  \epsilon \log \left( \Big( e^{(\mathbf{f}^{(l)}\oplus \mathbf{g}^{(l)})/\epsilon} \odot \mathbf{K} \Big) \mathbf{1} \right) + \mathbf{f}^{(l)}  ,  \frac{\Delta t}{1+\epsilon\Delta t}\gamma \right\}, \\
    &\mathbf{g}^{(l+1)} = \min \left\{ \epsilon\log(\mathbf{b}) -  \epsilon \log \left( \Big( e^{(\mathbf{f}^{(l+1)}\oplus \mathbf{g}^{(l)})/\epsilon} \odot \mathbf{K} \Big)^{\top} \mathbf{1} \right) + \mathbf{g}^{(l)}, \frac{\Delta t}{1+\epsilon\Delta t}\gamma \right\}.
\end{align}

Return $\mathbf{P}=e^{(\mathbf{f}^{(N)}\oplus \mathbf{g}^{(N)} - \mathbf{C})/\epsilon}$, assuming the iteration is stopped at $N$-th step.

\end{algorithm}

\subsubsection{sGW Solver: Algorithm and Convergence Analysis}

The full solver for sGW problem (\ref{eqn:sGW06}) is the combination of GVC in Algorithm \ref{alg:gvc}, Mirror-C iteration (\ref{eqn:MirrorC_01}), and the sOT solver in Algorithm \ref{alg:sot}. 

\begin{algorithm}

\caption{sGW solver}\label{alg:sgw}

Input: entropy coefficient $\epsilon$, threshold $\rho$, penalty $\gamma$, $\mathcal{M}$ as in (\ref{eqn:4thorder_tensor}).\\
Step I: Run GVC (Algorithm 1) to obtain a proper MVC $V_{\text{min}}^*$ and the $\mathcal{T}_{\mathcal{M}}^*$.\\
Step II: Initialization
\begin{itemize}
    \item $\mathbf{P}^{(0)}= \mathbf{a} \otimes \mathbf{b}$,
    \item $\mathbf{Q}^{(0)} = \exp \left( -\frac{\Delta t}{1+\epsilon\Delta t} \mathcal{M}\circ \mathbf{P}^{(0)} + \frac{1}{1+\epsilon\Delta t} \log \mathbf{P}^{(0)} \right) $,
    \item $Q^{(0)}_{ij} \leftarrow 0$ if $v_{ij} \in V_{\text{min}}^*$.
\end{itemize}

Step III: for $k=0,1,2,\cdots$, execute the following steps:

\begin{itemize}
\item update $\mathbf{P}^{(k+1)}$ by sOT solver (Algorithm \ref{alg:sot}),
\item update $\mathbf{Q}^{(k+1)} = \exp \left( -\frac{\Delta t}{1+\epsilon\Delta t} \mathcal{M}\circ \mathbf{P}^{(k+1)} + \frac{1}{1+\epsilon\Delta t} \log \mathbf{P}^{(k+1)} \right)$,
\item $Q^{(k+1)}_{ij} \leftarrow 0$ if $v_{ij} \in V_{\text{min}}^*$.
\end{itemize}
\end{algorithm}

The convergence of the sGW solver in Algorithm \ref{alg:sgw} consists of two parts, one is the convergence of the sOT solver in Algorithm \ref{alg:sot}, and the other is the convergence of the Mirror-C update in (\ref{eqn:MirrorC_01}).
The convergence of the sOT solver has been well established in \cite{peyre2015entropic}. Indeed, \cite{peyre2015entropic} proves the convergence of the Dykstra iteration when applying to a more general minimization problem:
\begin{align}\label{eqn:Dykstra_general}
\min_{\mathbf{P}\in\mathbb{R}_+^{n\times m}} B_g(\mathbf{P}|\mathbf{K}) + \hat{h}_1(\mathbf{P}) + \hat{h}_2(\mathbf{P}).
\end{align}
Here $g$ is a given proper closed and strictly convex and differentiable function, and $B_g$ is the Bregman divergence defined as 
\[
B_g(\mathbf{P}|\mathbf{Q}) = g(\mathbf{P}) - g(\mathbf{Q}) - \langle \nabla g(\mathbf{Q}), \mathbf{P} - \mathbf{Q} \rangle.
\]
In Proposition 3.1 of \cite{peyre2015entropic}, it proves that if the following condition holds
\begin{align*}
\operatorname{ri}(\operatorname{dom}(\hat{h}_1)) \cap \operatorname{ri}(\operatorname{dom}(\hat{h}_2)) \cap \operatorname{ri}(\operatorname{dom}(g)) \neq \emptyset \operatorname{~},
\end{align*}
then the Dykstra iterate $\mathbf{P}^{(k)}$ converges to the solution of the general minimization problem (\ref{eqn:Dykstra_general}).

Since $B_g = \text{KL}$ when $g$ is the entropy function, by taking $\hat{h}_1$ and $\hat{h}_2$ as $h_1$ and $h_2$ in (\ref{eqn:MirrorC_01}), respectively, we can easily verify that the intersection of three relative interiors is exactly the nonempty admissible set
\begin{align}
\operatorname{ri}(\operatorname{dom}(\hat{h}_1)) \cap \operatorname{ri}(\operatorname{dom}(\hat{h}_2)) \cap \operatorname{ri}(\operatorname{dom}(g)) = 
    \mathbf{U}(\leq \mathbf{a}, \leq \mathbf{b}) \cap M_{\mathbf{C}}^+  
\end{align}
where $\mathbf{C} = -\log \mathbf{Q}^{(k)}$ with the $\infty$-pattern specified by $\mathcal{T}_{\mathcal{M}}^*$. Therefore, the sOT solver in the inner loop converges.

The convergence of the KL Mirror-C descent for sGW problem is due to \cite{boct2016inertial}. Given
$
\mathcal{J}(\mathbf{P}) = \frac{1}{2}\langle \mathcal{M}, \mathbf{P} \otimes \mathbf{P}\rangle_F
$
and the admissible set $\mathbf{U}(\leq \mathbf{a}, \leq \mathbf{b}) \cap \mathcal{T}_{\mathcal{M}}^*$. The Lipschitz constant of $\nabla \mathcal{J}(\mathbf{P})$ is defined by $ L_{\nabla \mathcal{J}} =\sup \frac{\|\nabla \mathcal{J}(\mathbf{P}^1)- \nabla \mathcal{J}(\mathbf{P}^2\|_F} {\|\mathbf{P}^1-\mathbf{P}^2 \|_F}$ for $\mathbf{P}^1,\mathbf{P}^2 \in \mathbf{U}(\leq \mathbf{a}, \leq \mathbf{b}) \cap \mathcal{T}_{\mathcal{M}}^* $ and $ \mathbf{P}^1 \neq \mathbf{P}^2.$ Notice that 
    \begin{align*}
       \|\nabla \mathcal{J}(\mathbf{P}^1)- \nabla \mathcal{J}(\mathbf{P}^2) \|_F^2 
        &=  \| \mathcal{M} \circ (\mathbf{P}^1 - \mathbf{P}^2) \|_F^2 \\
       &=  \sum_{i=1}^n \sum_{j=1}^m \left( \sum_{k=1}^n \sum_{l=1}^m \mathcal{M}_{i j k l} (P^1_{k l}-P^2_{k l}) \right)^2\\
       &\leq  \left( \max (\max_{i,k} D^1_{ik}, \max_{j,l} D^2_{jl}) \right)^4 \sum_{i=1}^n \sum_{j=1}^m \left( nm \sum_{k=1}^n \sum_{l=1}^m (P^1_{k l}-P^2_{k l})^2 \right)\\
       &=  \left( \max (\max_{i,k} D^1_{ik}, \max_{j,l} D^2_{jl}) \right)^4 n^2 m^2 \| \mathbf{P}^1-\mathbf{P}^2 \|_F^2.
    \end{align*}
Therefore, $L_{\nabla \mathcal{J}}$ = $ nm \left( \max (\max_{ik} D^1_{ik}, \max_{jl} D^2_{jl}) \right)^2.$ On the other hand, the entropy function $ y = x(\ln x - 1)$ is $\frac{1}{N}$-strongly convex when $x \in (0, N)$, so $-H(\mathbf{P}) = \langle \mathbf{P}, \log \mathbf{P}-\mathbf{1} \rangle_F$ is $\frac{1}{N}$-strongly convex with $N=\min(\|a\|_{\infty}, \|b\|_{\infty} )$ because $\mathbf{P} \in \mathbf{U}(\leq \mathbf{a}, \leq \mathbf{b}) \cap \mathcal{T}_{\mathcal{M}}^*$ implies that $\mathbf{P}_{ij} \in (0, \min(\|\mathbf{a}\|_{\infty}, \|\mathbf{b}\|_{\infty} ))$.

Adopting the notation in \cite{solomon2016entropic} $\eta = \frac{\epsilon \Delta t}{1+\epsilon \Delta t}$, the Mirror-C update (\ref{eqn:MirrorC_01}) becomes
\begin{align}\label{eqn:MirrorC_02}
\mathbf{P}^{(k+1)} = \underset{\mathbf{P}\in \mathcal{T}_\mathcal{M}^*}{\mathrm{argmin}} \  \text{KL} \left(\mathbf{P} \Big| \left[e^{-\mathcal{M}\circ\mathbf{P}^{(k)}/\epsilon}\right]^{\eta} \odot \left[\mathbf{P}^{(n)}\right]^{(1-\eta)} \right) + h_1(\mathbf{P1}) + h_2(\mathbf{P}^{\top}\mathbf{1}).
\end{align}
Then following Lemma 6 in \cite{boct2016inertial}, to guarantee the convergence of KL Mirror-C descent for non-convex $\mathcal{J}(\mathbf{P})+\mathcal{K}(\mathbf{P})$, we need 
    \begin{align*}
        \frac{1}{N} - \frac{\eta}{\varepsilon (1-\eta)} \cdot L_{\nabla \mathcal{J}} > 0 \iff  
        \frac{\eta}{\varepsilon (1-\eta)} < \frac{1}{N \cdot L_{\nabla \mathcal{J}} } \iff
        \eta < \frac{\varepsilon}{N \cdot L_{\nabla \mathcal{J}} + \varepsilon },
    \end{align*}
or equivalently $\Delta t \le \frac{1}{N\cdot L_{\nabla\mathcal{J}}}$.
We summarize the convergence result in the following theorem.
\begin{theorem}\label{theorem:sGW_conv}
Let $\{ \mathbf{P}^{(k)} \}_{k\ge 0}$ be generated by the sGW solver in Algorithm \ref{alg:sgw}. Then it converges to a critical point of $\mathcal{J}(\mathbf{P})+ \mathcal{K}(\mathbf{P})$ when $\eta < \frac{\varepsilon}{N \cdot L_{\nabla \mathcal{J}} + \varepsilon }$, or equivalently $\Delta t \le \frac{1}{N\cdot L_{\nabla\mathcal{J}}}$, where $N = \min \{ \|\mathbf{a}\|_{\infty}, \|\mathbf{b}\|_{\infty} \}$ and $L_{\nabla\mathcal{J}} = nm \max_{ijkl}\{ D^1_{ik}, D^2_{jl}\}^2$.
\end{theorem}

\subsection{Downsampling and Recovering Coupling Between Full Datasets From Downsampled Data}
To deal with large datasets, we first compute sGW on downsampled data and later recover the full coupling by solving an sOT problem. For downsampling, we consider two shape-preserving approaches, Mapper \cite{singh2007topological} and Geosketch \cite{hie2019geometric} which highlight topology conservation and geometry conservation, respectively. For a dataset often in a reduced-dimensional space, Mapper first constructs a cover of the space by overlapping hypercubes. While Mapper can handle more general spaces like a circle, we restrict ourselves to Euclidean spaces here. Then clustering is performed for data within each hypercube and the resulting clusters are connected based on their shared data points. Here, we take the collection of the centroid of each cluster as a downsampling of the original data. Geosketch divides the space into boxes which induces a cover of the dataset. It then randomly picks one data point in each box to form a downsampled dataset. 

Consider two datasets each with $n$ and $m$ points, and an optimal sGW transport plan computed for the corresponding downsampled datasets, $\mathbf{\hat{P}}\in\mathbb{R}^{\hat{n}\times \hat{m}}$, we aim to recover the coupling $\mathbf{P}\in\mathbb{R}^{n\times m}$ between the two full datasets. Let $\mathbf{\hat{D}}^1\in\mathbb{R}^{n\times\hat{n}}$ and $\mathbf{\hat{D}}^2\in\mathbb{R}^{m\times\hat{m}}$ be the geodesic distances between all data points and the downsampled data points of the two datasets. Each row describes the location of a data point with respect to the downsampled points and we use $\mathbf{\hat{P}}$ to bridge the two location descriptions. Specifically, we construct a cost matrix $\mathbf{C}^1_{ij}=d_p(\mathbf{\hat{D}}^1_{i\cdot}, (\mathbf{\hat{D}}^{2}\mathbf{\hat{P}'}_c)_{j\cdot})$ where $\mathbf{\hat{P}'}_c$ is the transpose of column-normalized $\mathbf{\hat{P}}$ (columns sum to $1$). Similarly, by mapping $\mathbf{\hat{D}}^1$ to the space of $\mathbf{\hat{D}}^2$, we construct $\mathbf{C}^2_{ij}=d_p((\mathbf{\hat{D}}^{1}\mathbf{\hat{P}}_r)_{i\cdot}, \mathbf{\hat{D}}^2_{j\cdot})$, where $\mathbf{\hat{P}}_r$ is the row-normalized $\mathbf{\hat{P}}$. Then, we obtain a full coupling between the two datasets by solving an sOT problem with $(\mathbf{C}^1+\mathbf{C}^2)$ as the cost.

\section{Numerical Experiments}\label{sec:numerical}

\begin{figure}[t]
\centerline{
\includegraphics[width=0.8\textwidth]{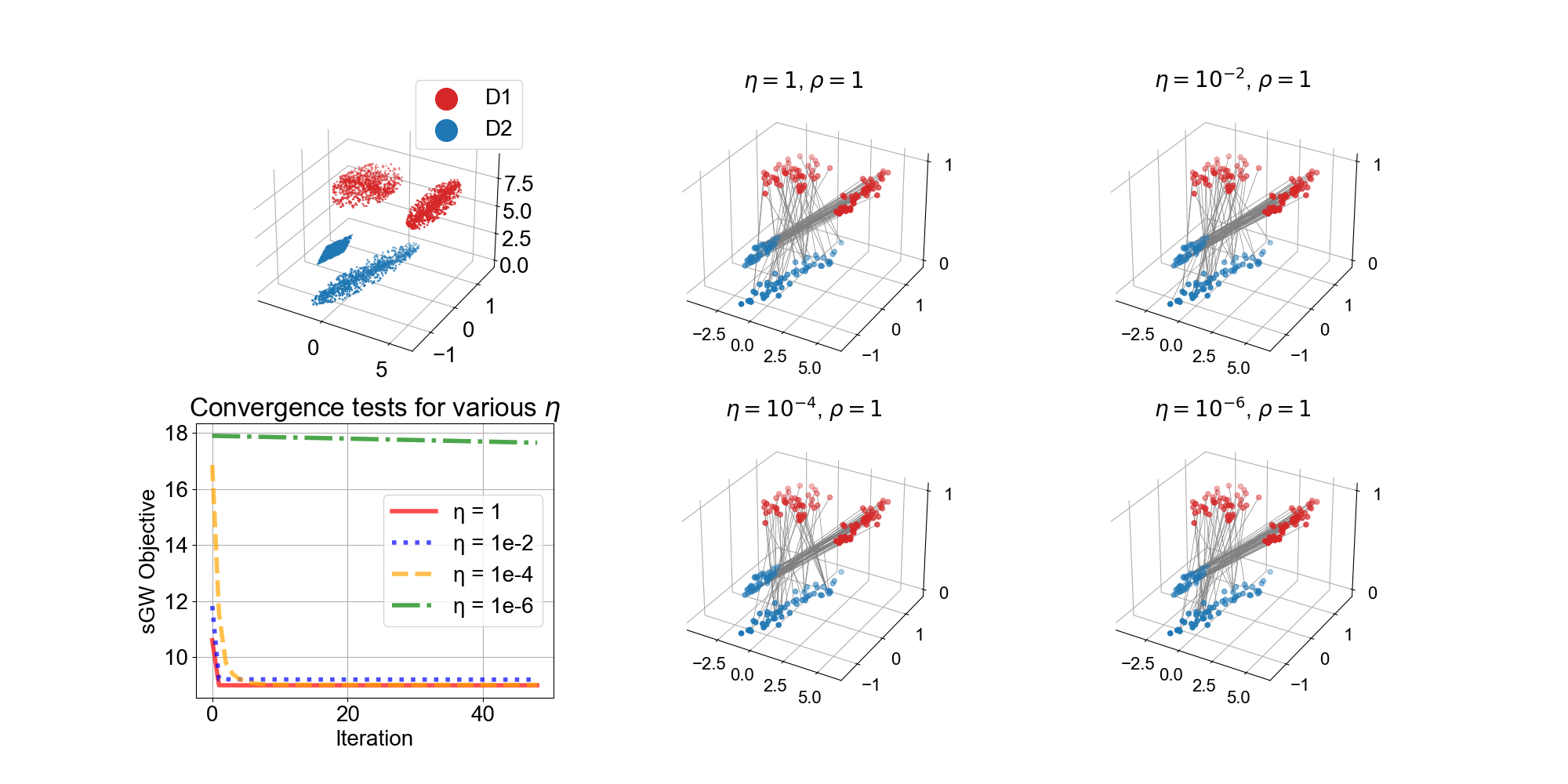}
}
\vspace{-0.25in}
\caption{The impact of the step size $\eta$ on sGW matching on synthetic point-cloud data.}
\label{fig:3dexample_1_eta}
\end{figure}

In this section, we first use the sGW solver on simple synthetic data to validate our proposed sGW model framework. Then we apply the model to some real applications in single-cell RNA sequencing (scRNA-seq) data. In all the numerical experiments, we take the marginal distributions $\mathbf{a}$ and $\mathbf{b}$ to be uniform as given in Eq. (\ref{eqn:uniform_marginal}).

\subsection{Synthetic Point-cloud Matching}

\begin{figure}[h]
\centerline{
\includegraphics[width=0.8\textwidth]{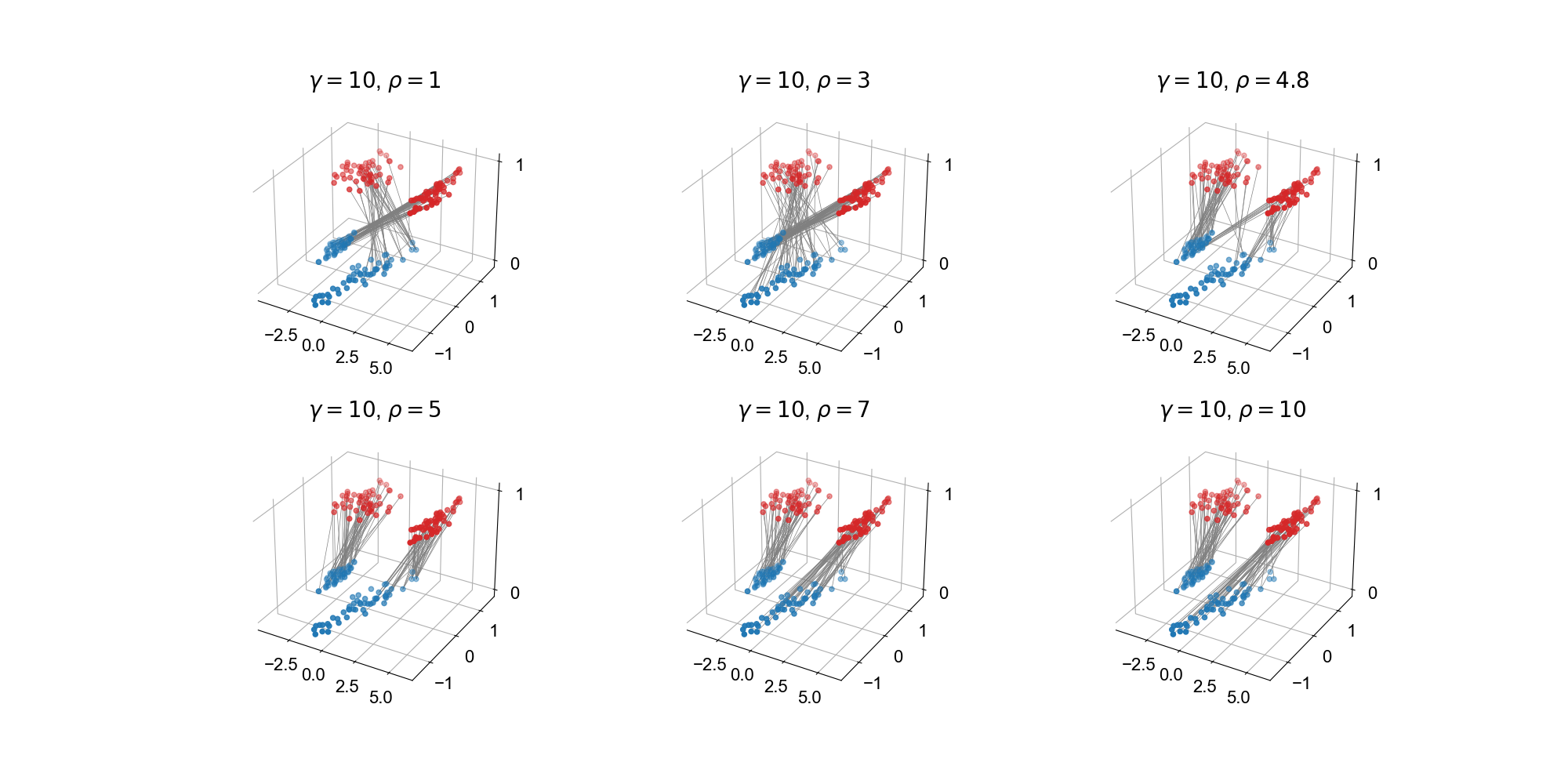}
}
\vspace{-0.25in}
\caption{The impact of the shreshold $\rho$ on thesGW matching on synthetic point-cloud data.}
\label{fig:3dexample_2_rho}
\end{figure}

In this section, we present a numerical result for the matching between two point clouds as in \cite{sejourne2021unbalanced}. Let $E_2, E_3$ be the uniform distributions on 2D and 3D ellipses, respectively, and $C$ and $S$ be uniform distributions over a 2D square and a 3D sphere, respectively. We take the two point clouds $\mathbf{a}= 0.5E_3 + 0.5S$ and $\mathbf{b}= 0.3E_2+0.7C$.

In the first numerical experiment presented in Fig. \ref{fig:3dexample_1_eta}, we fix the threshold $\rho = 1$, and study the impact of the step size $\eta = \frac{\epsilon \Delta t}{1 + \epsilon \Delta t}$ on the sGW matching. When $\eta = 1, 10^{-2}, 10^{-4}, 10^{-6}$, the sGW matching is similar to each other. Besides, the convergence test in the bottom left corner of Fig. \ref{fig:3dexample_1_eta} indicates that the larger $\eta$ is, the faster the convergence is. This phenomenon is consistent with the numerical finding in \cite{solomon2016entropic}. Therefore, from now on, though it is possibly beyond the convergence condition in Theorem \ref{theorem:sGW_conv}, we will still take $\eta = 1$ for computational efficiency.

The second numerical experiment in Fig. \ref{fig:3dexample_2_rho} illustrates the effect of the threshold $\rho$ on the sGW matching. Here we fix $\eta = 1$ and $\gamma = 10$. Note that when $\rho$ is small, there are few matching points in $\mathbf{a}$ and $\mathbf{b}$. As $\rho$ becomes greater, more and more points are matched. When $\rho = 3$, all the points are matched but in a crossed manner. When $\rho$ further grows, the optimal solution tends to become straightforward matching between point clouds. When $\rho = 10$, which exceeds its upper bound $\max_{ijkl}\sqrt{\mathcal{M}_{ijkl}}$, all the points are matched straightfowardly. The last result reduces to the PGW matching with $s = \min\{\|\mathbf{a}\|_1, \|\mathbf{b}\|_1  \}$. Since here we take $\|\mathbf{a}\|_1 = \|\mathbf{b}\|_1$, it is also identical to the original GW matching.

In the third numerical experiment in Fig. \ref{fig:3dexample_3_gamma}, we study the influence of the penalty $\gamma$ on the sGW matching. By taking a sequence of increasing $\gamma = 0, 0.01, 0.1, 1, 10, 100$, we note that the total transported mass $\sum_{ij} P_{ij}$ keeps increasing. Until $\gamma$ exceeds some certain threshold (say, $\gamma = 10$ in this example), the total transported mass becomes unchanged. This numerical experiment validates our definition of sGW in (\ref{eqn:sGW04}). Indeed, this example indicates that $\gamma$ does not have to be too large to make the sGW work. From now on, we will fix $\gamma = 10$.

\begin{figure}[h]
\centerline{
\includegraphics[width=0.8\textwidth]{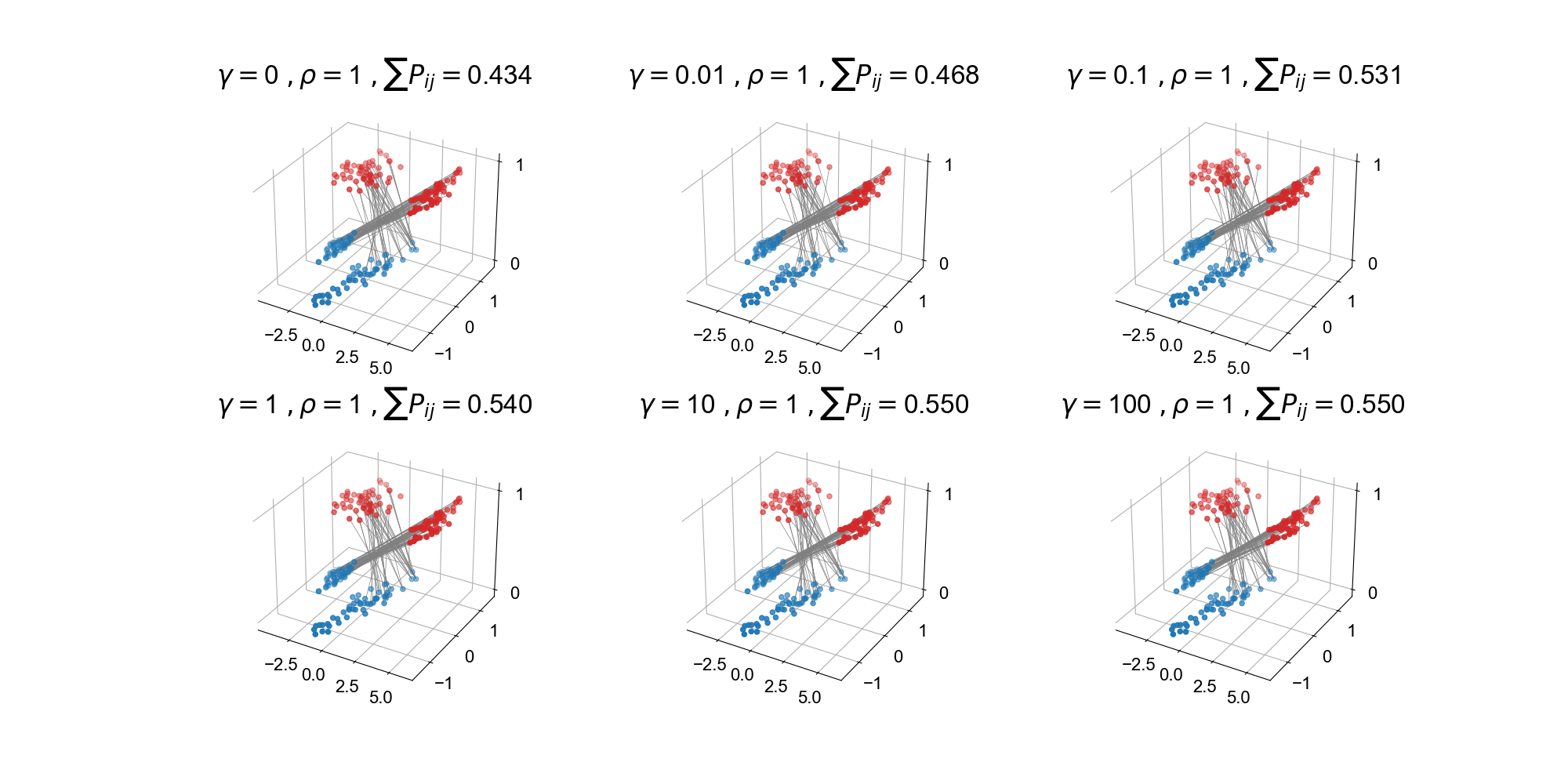}
}
\vspace{-0.25in}
\caption{The $\gamma$-effect on the sGW matching on synthetic point-cloud data.}
\label{fig:3dexample_3_gamma}
\end{figure}

\subsection{Calculation of MVC}

In this section, we conduct numerical tests to assess the matching between two smaller datasets and illustrate the complexity of calculating the MVC. We consider dataset 1, consisting of 5 points in 2D, and dataset 2, consisting of 9 points in 3D. The step size $\eta = 1$ and $\rho = 0.5$ are fixed.

In Fig. \ref{fig:cross_example}, in the top left subfigure, we generate two matchings as illustrated in the last two subfigures in the top row. The presence of multiple solutions arises from the mirror symmetry of the datasets. Similarly, for the datasets in the bottom row, we also observe two different solutions—one representing the most natural straightforward matching and the other arising from rotational symmetry. 
In the second column of Fig. \ref{fig:cross_example}, we present the intricate graph structure used for the MVC calculation. In this example, the vertex set $V$ consists of 45 points. By implementing the GVC algorithm as described in Algorithm \ref{alg:gvc}, we identify 40 vertices in the outcome, denoted by $V_{\min}$. Because the datasets are small in this example, we can manually calculate the MVC, allowing us to compare the sizes of the exact MVC and $V_{\min}$ produced by the GVC algorithm. In general, assessing the approximation accuracy of the GVC algorithm for such synthetic graphs can be challenging, particularly when dealing with distributions $\mathbf{a}$ and $\mathbf{b}$ over larger point clouds.


\begin{figure}[h]
\centerline{
\includegraphics[width=0.8\textwidth]{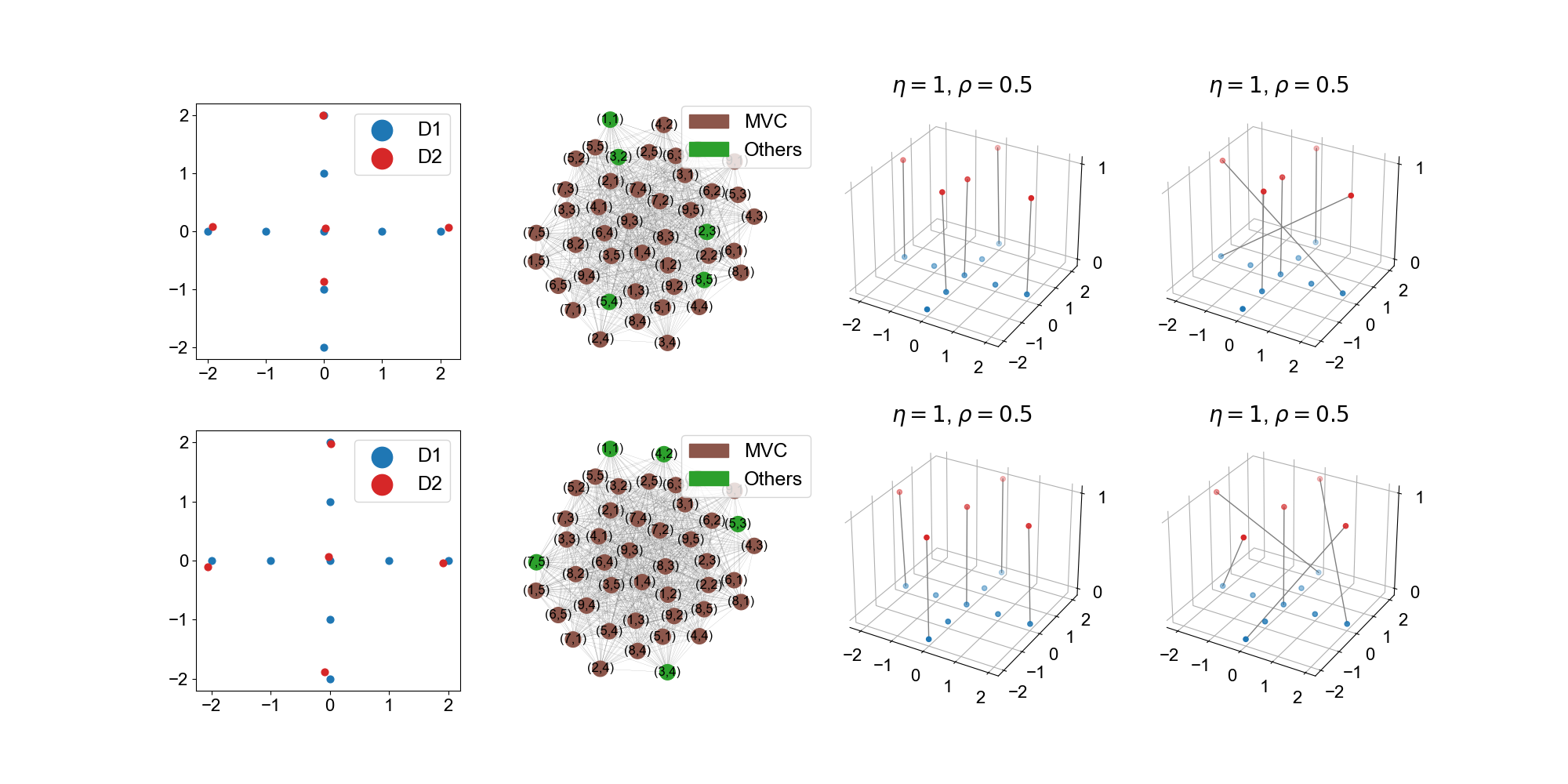}
}
\vspace{-0.25in}
\caption{sGW matching for smaller datasets and the intricate graph structure in the MVC calculation.}
\label{fig:cross_example}
\end{figure}

\subsection{Surface Matching}
To assess the performance of sGW in practical applications, we conducted a comparison of surface matching between two octopuses. Our objective is to explore the trade-off between strict geometric similarity and complete mass matching, denoted as $s = \sum \mathbf{P}_{ij}$. We employed Mapper for downsampling, resulting in octopus 1 having 144 mesh points and octopus 2 containing 127. Distance matrices were calculated using geodesic distances on the full dataset and the mesh points in octopus 1 are colored based on their position (8 legs and 1 body). In this case, $\rho_{\max} = 2.04$ and Fig. \ref{fig:oc1} displays the matchings via sGW for $\rho = 0.2, 1.5, 2.0$. At a lower threshold ($\rho = 0.2$), the matching is more rigid, adhering closely to geometric similarity. Here, the coupling matrix $\mathbf{P}$ typically has only one non-zero element per row and column. As $\rho$ increases, the coverage expands and $s$ grows. However, this leads to more fragmented matching ($\rho =2.0$), and we assign to vertex index $k$ in octopus 2 the same color as vertex index $i$ in octopus 1, where $k= \operatorname{argmax}_j \mathbf{P}_{ij}$.

We further evaluated the matching between a filtered version of octopus 1 and the original octopus 2, as shown in Fig. \ref{fig:oc2}. This filtered octopus 1 was created by removing eight half legs in total. The new downsampling of the filtered octopus 1, using Mapper, results in 77 points, while octopus 2 remains at 127 points. The geodesic distance matrices for the modified octopus 1 and octopus 2 range from (0, 1.20) and (0, 2.04), respectively. Similar to the observation in Fig. \ref{fig:3dexample_2_rho}, the matching region of octopus 2 in Fig. \ref{fig:oc2} is only partial, covering half of its surface, and it never encompasses the entire octopus as long as $\rho \leq 1.20$. As $\rho$ increases, the preferred solution gradually shifts towards a regular matching that covers the entire octopus 2. Different from sGW, the patterns in GW, PGW, and UGW always spread across the whole octopus 2. This unique aspect of sGW, particularly noticeable when $\rho$ is relatively small, highlights its ability to respect spatial information while providing maximum geometric similarity.

\begin{figure}[h]
\centerline{
\includegraphics[width=0.6\textwidth]{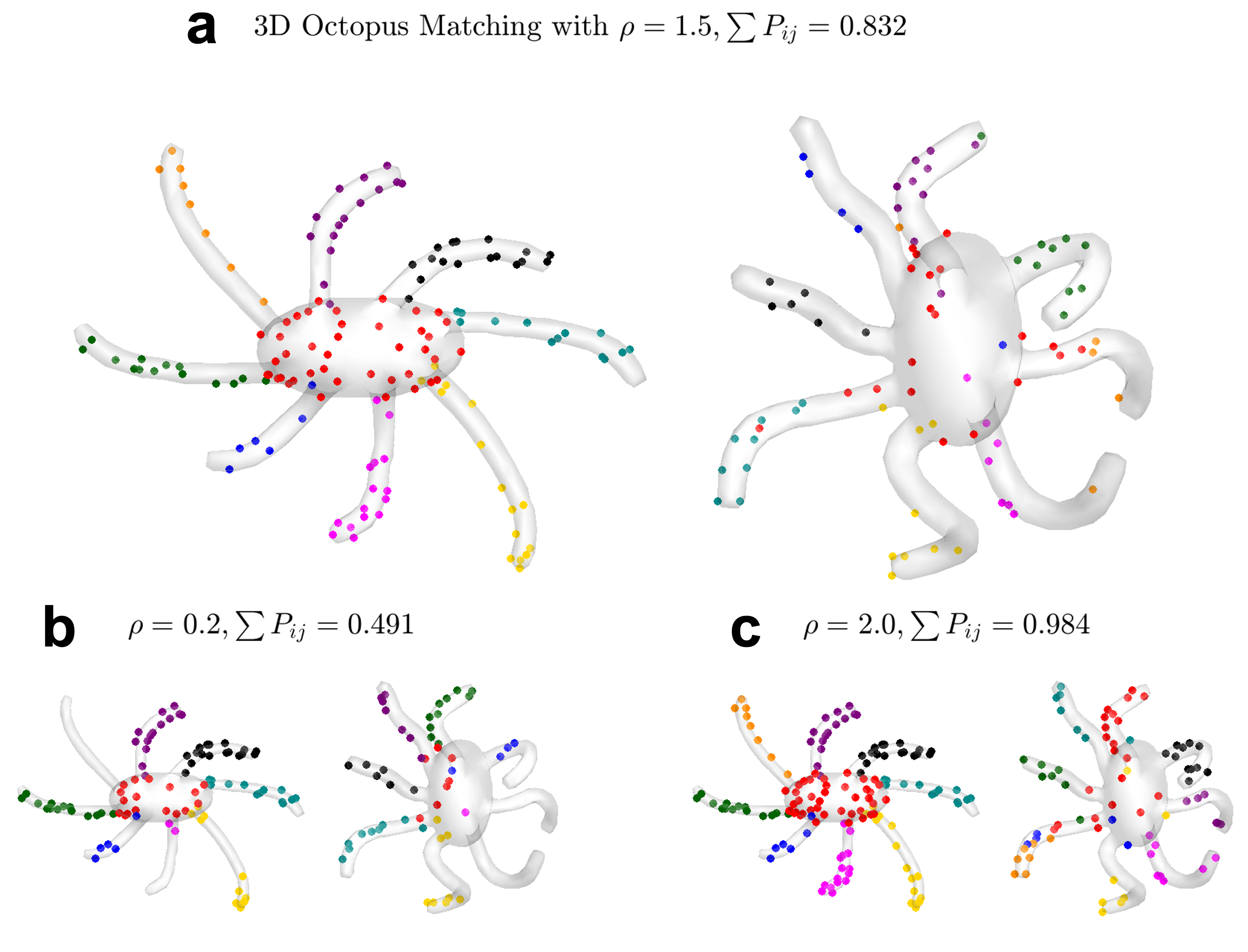}
}
\vspace{-0.25in}
\caption{sGW matching between full surfaces of two octopuses.}
\label{fig:oc1}
\end{figure}

\begin{figure}[h]
\centerline{
\includegraphics[width=0.6\textwidth]{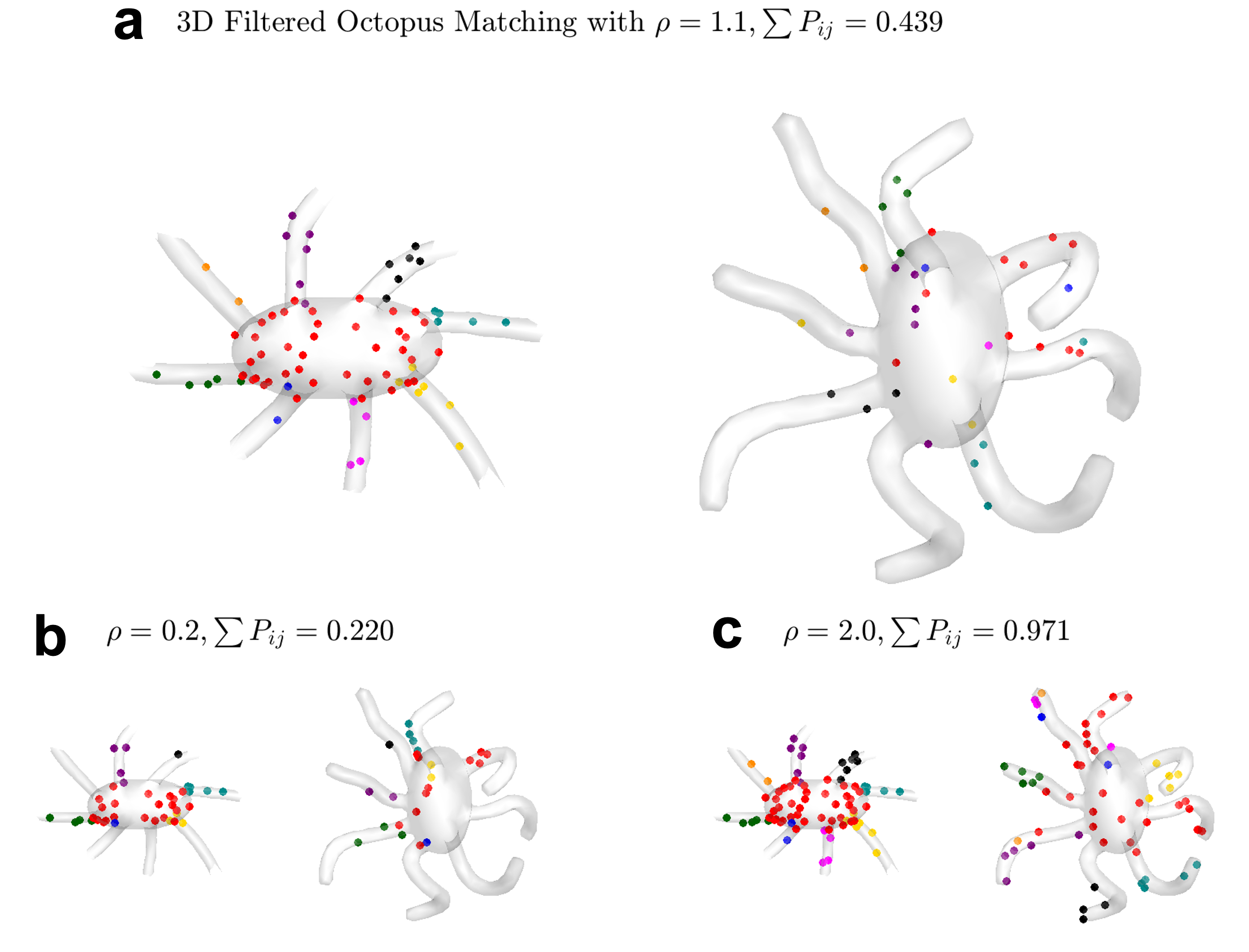}
}
\vspace{-0.25in}
\caption{sGW matching between the filtered surface of octopus 1 and the full surface of octopus 2.}
\label{fig:oc2}
\end{figure}

\subsection{Single-cell RNA Sequencing Data Matching}

In this section, we evaluate sGW on single-cell RNA sequencing (scRNA-seq) \cite{svensson2018exponential} data resulting from a revolutionary technology that examines gene expression profiles of tissues with unprecedented single-cell resolution. We first test on a simulated \cite{saelens2019comparison} and a real scRNA-seq dataset \cite{han2018mapping}. Principal component analysis is first performed on the original high-dimensional data. A $k$-nearest-neighbor graph is then constructed based on the Euclidean distance in the PCA embedding. The shortest path lengths along the graph with respect to the Euclidean distance in PCA embedding are used as the cost matrices for sGW. A TSNE dimension reduction on the PCA embedding is performed for visualizing the data. Correct couplings are obtained using sGW on several cases including one dataset being a subset of the other, partially overlapping datasets with large or small overlap (Figs.\ref{fig:bioexample_1_toy} and \ref{fig:bioexample_2_real}).

\begin{figure}[h]
\centerline{
\includegraphics[width=0.7\textwidth]{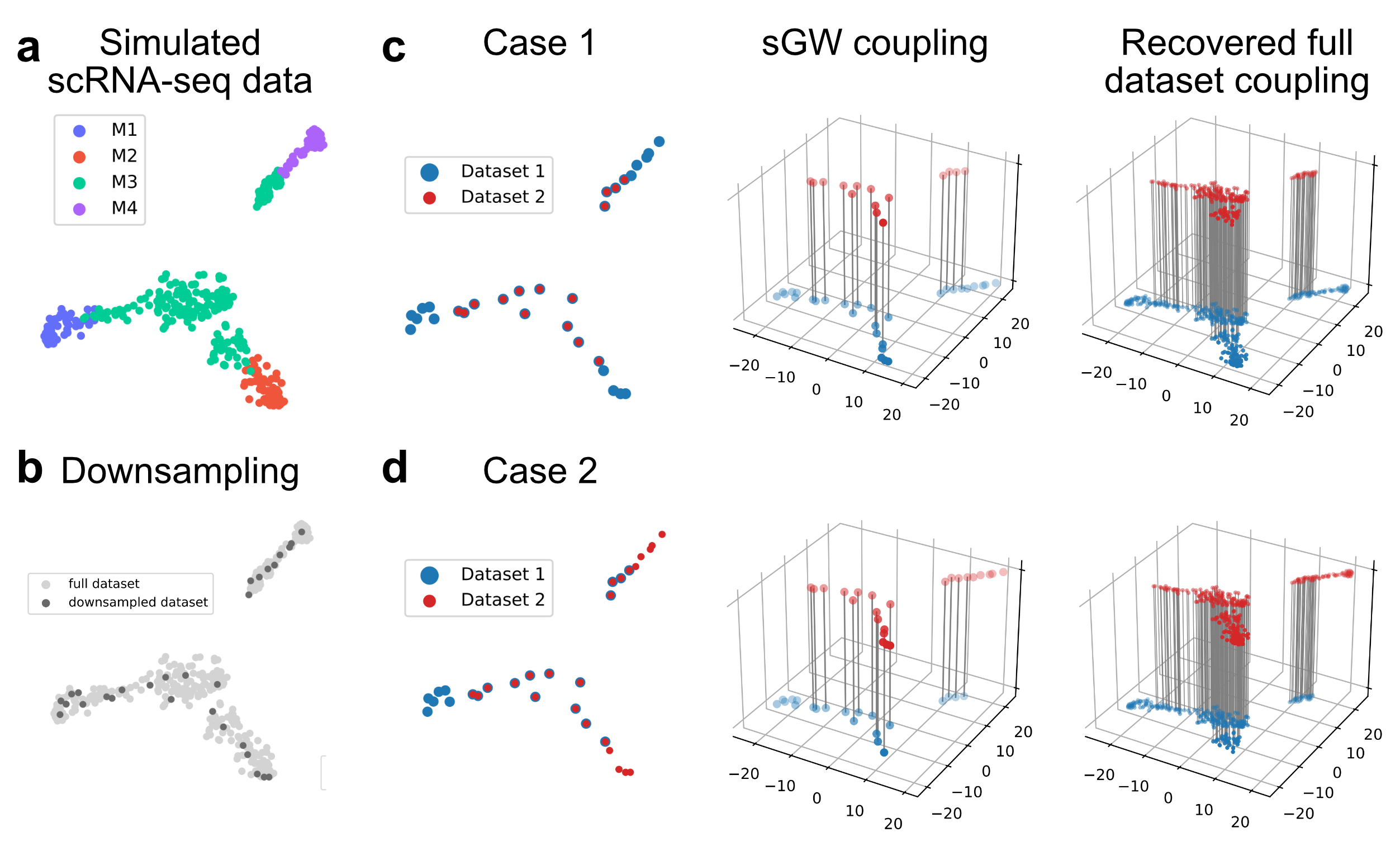}
}
\caption{sGW matching on a simulated scRNA-seq data. ($\mathbf{a}$) The TSNE visualization of the simulated scRNA-seq data colored by cell type; ($\mathbf{b}$) Downsampling of the full data; ($\mathbf{c}$) Case 1: dataset 2 is a subset of dataset 1, the sGW coupling between the downsampled data and the recovered full dataset coupling by sOT; ($\mathbf{d}$) Case 2: two partially overlapping datasets.}
\label{fig:bioexample_1_toy}
\end{figure}

\begin{figure}[h]
\includegraphics[width=0.7\textwidth]{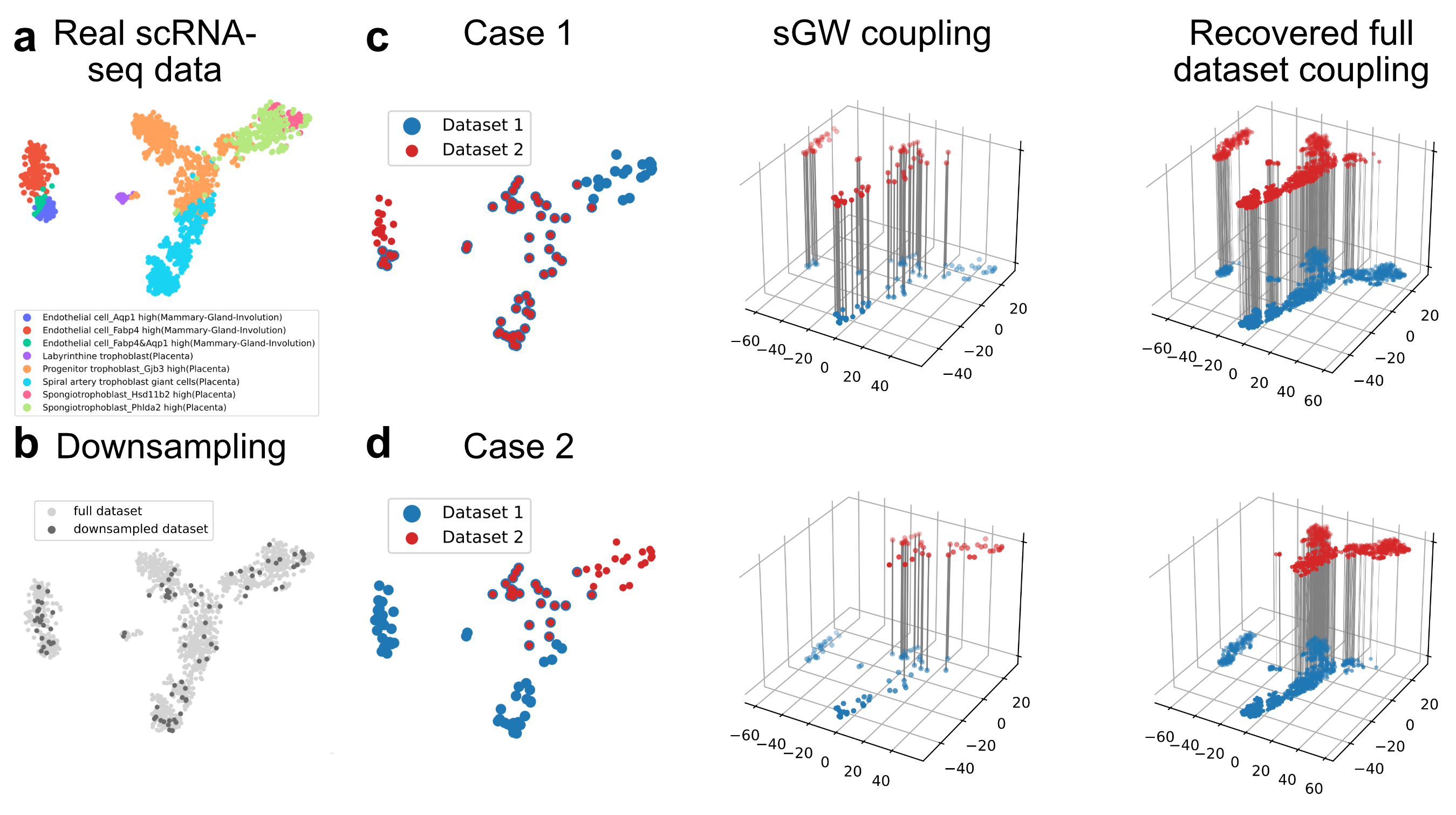}
\centering
\caption{sGW matching on a real scRNA-seq data. (\textbf{a}) The TSNE visualization of the scRNA-seq data colored by cell type; (\textbf{b}) Downsampling of the full data; (\textbf{c}) Case 1: two partially overlapping datasets with large overlap, the sGW coupling between the downsampled data and the recovered full dataset coupling by sOT; (\textbf{d}) Case 2: two partially overlapping datasets with small overlap.}
\label{fig:bioexample_2_real}
\end{figure}

\begin{figure}[h]
\includegraphics[width=0.6\textwidth]{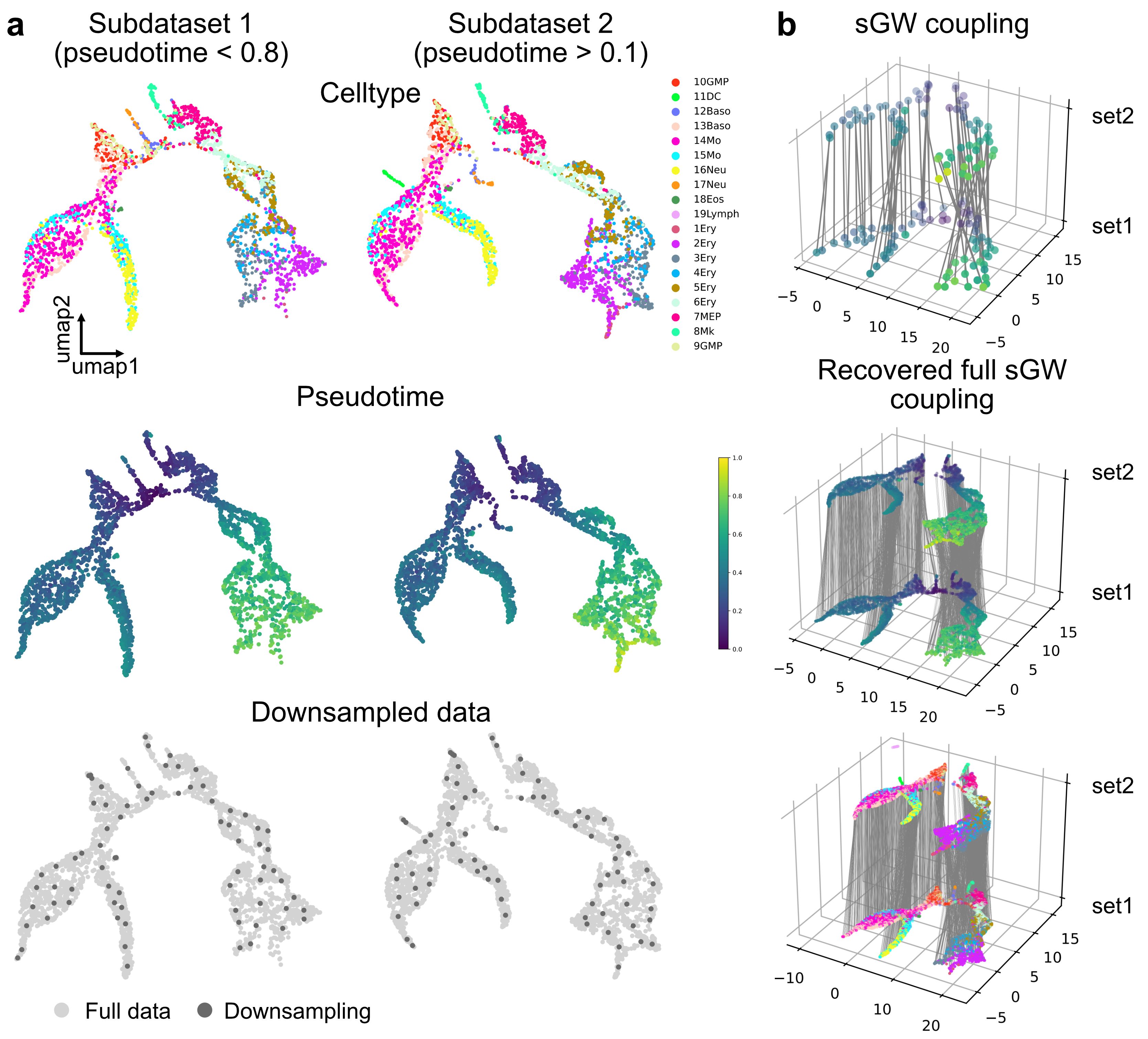}
\centering
\caption{sGW matching of different embeddings of a real scRNA-seq data. (\textbf{a}) Two UMAP embeddings of two subsets of the scRNA-seq data with colors showing cell types, pseudotime, and downsampling; (\textbf{b}) Independent 2-dimensional UMAP and 3-dimensional UMAP embeddings of the same datasets as in (a). A correct coupling should match the pseudotime or cell types represented by the point colors.}
\label{fig:bioexample_3_real}
\end{figure}

\begin{figure}[h]
\includegraphics[width=0.7\textwidth]{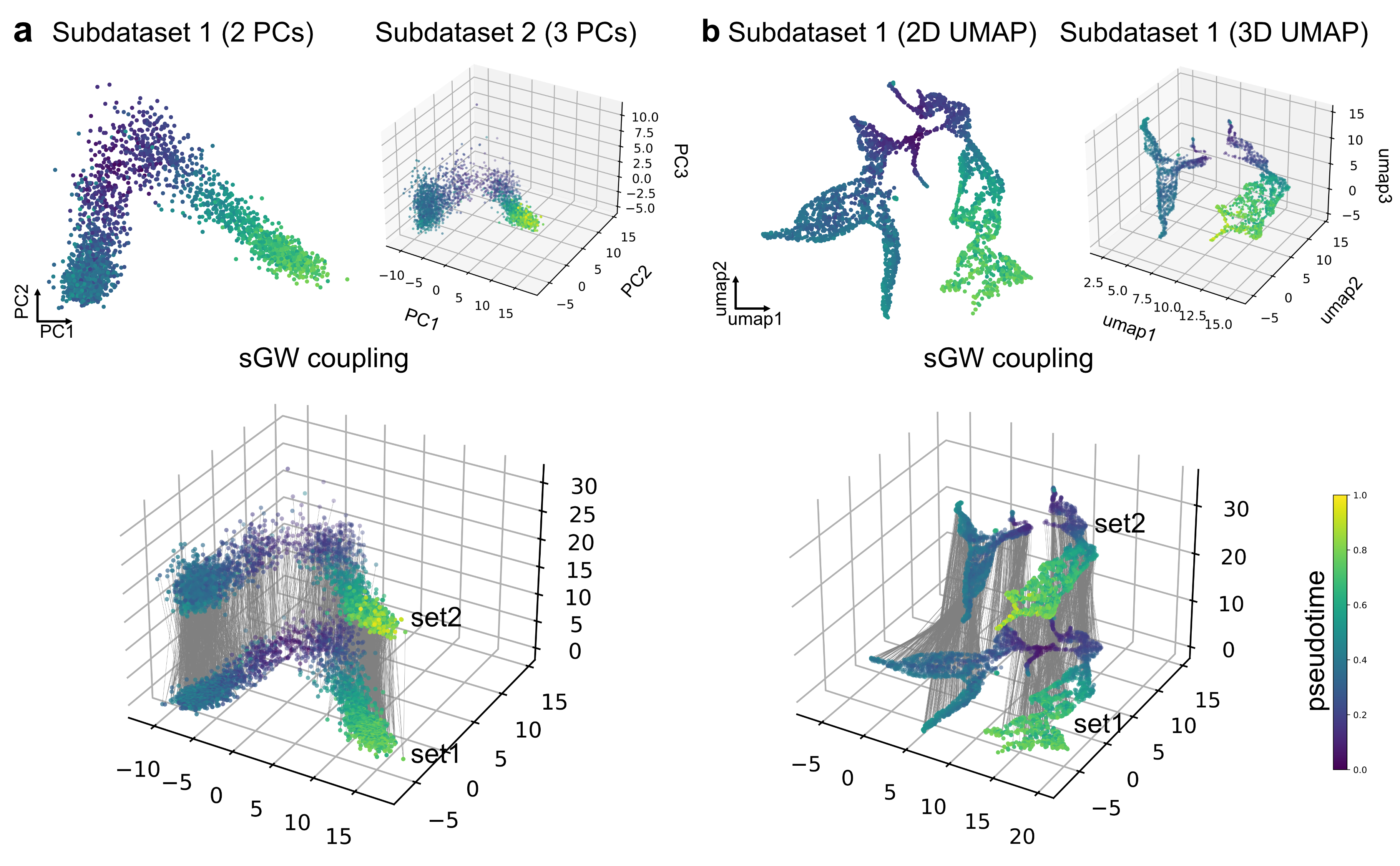}
\centering
\caption{sGW matching of different dimensional embeddings of a real scRNA-seq data. (\textbf{a}) A 2-dimensional and a 3-dimensional embedding of the dataset by taking the first two or three principal components, and the resulting sGW coupling; (\textbf{b}) The sGW coupling of the downsampled data and the recovered full dataset coupling by sOT. A correct coupling should match the pseudotime of cells represented by colors.}
\label{fig:bioexample_3.5_real}
\end{figure}

\begin{figure}[h]
\includegraphics[width=0.95\textwidth]{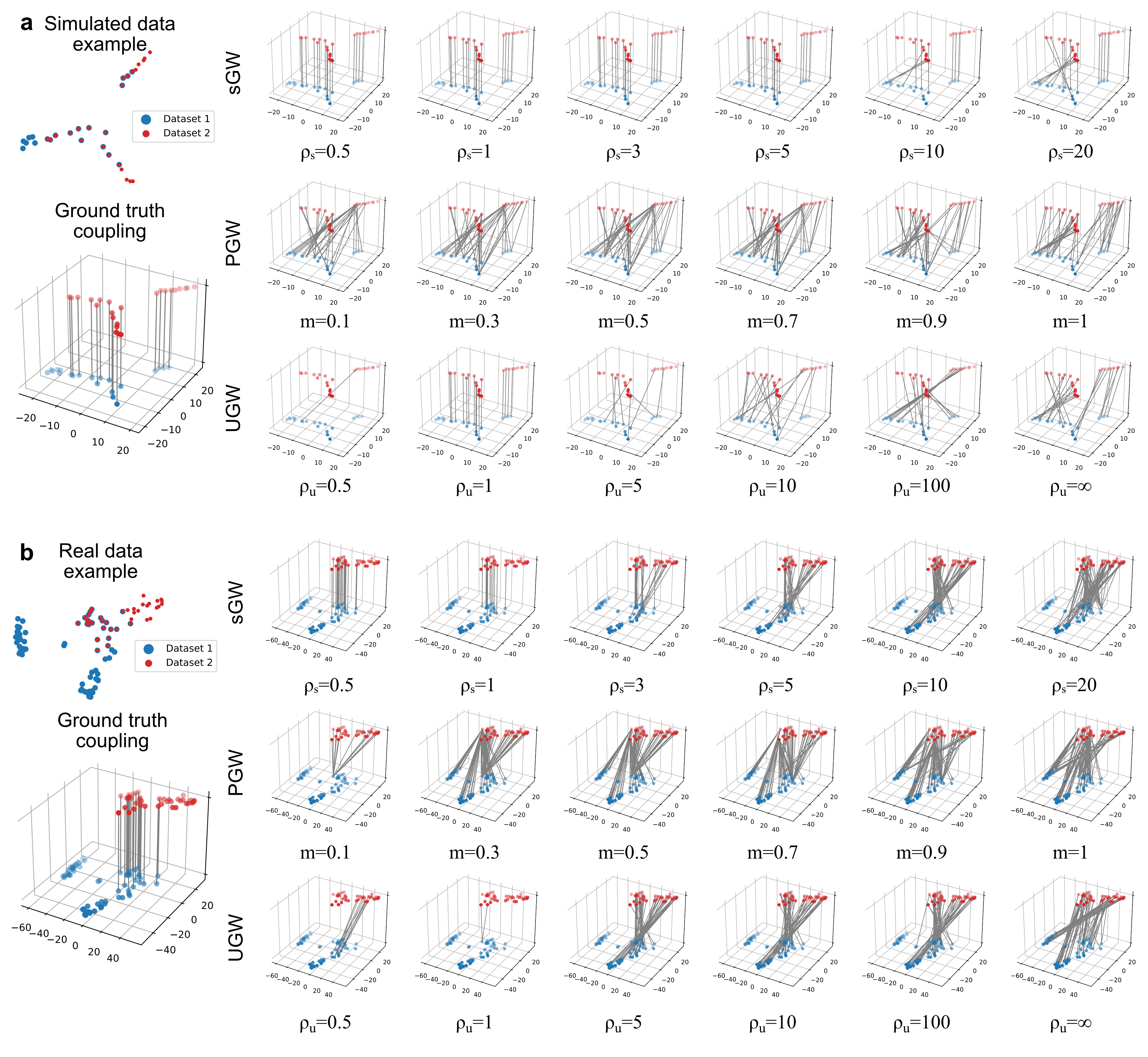}
\centering
\caption{Comparison of sGW to partial GW (PGW) and unbalanced GW (UGW) on two scRNA-seq datasets. (\textbf{a}) For the simulated scRNA-seq data, the results from sGW, PGW, and UGW using different values for the key parameter of each method, the cutoff value $\rho_s$ for sGW, the total transported mass $m$ for PGW, and the KL divergence coefficient $\rho_u$ for UGW; (\textbf{b}) The GW results on the real scRNA-seq data.}
\label{fig:bioexample_4_comparison}
\end{figure}

We also test the performance of sGW on different embeddings of a dataset. Here, we use the scRNA-seq data of myeloid progenitors \cite{paul2015transcriptional}. Here, we perform UMAP dimension reduction on the PCA space of the dataset with different random initializations resulting in different embeddings of the two subsets (Fig.\ref{fig:bioexample_3_real}a). The shortest path distances along the $k$-nearest-neighbor graph induced by the distance in the UMAP embeddings are used as cost matrices. The results show that sGW also derives generally correct coupling of data on different embeddings including independently derived same dimensional UMAP embeddings (Fig.\ref{fig:bioexample_3_real}b) and embeddings of different dimensions (Fig.\ref{fig:bioexample_3.5_real}).

Using the downsampled datasets in the previous examples, we demonstrate the differences among sGW, PGW, and UGW. We conduct experiments with different values of the key parameters of sGW, PGW, and UGW. In sGW, we tested with different values of the cost cutoff threshold $\rho = \rho_s$. In PGW, different values of total transported mass $m$ are tested. In UGW, different values are used for the coefficient $\rho_u$ of the KL divergence terms $\text{KL}(\mathbf{P1}|\mathbf{a})$ and $\text{KL}(\mathbf{P}^\mathrm{T}\mathbf{1}|\mathbf{b})$. In the experiments, uniform distributions of data are used and all distance matrices are normalized by dividing by the maximum value of $\mathbf{D}^1$. Entropy regularization is used for all formulations with the coefficient $\epsilon=0.1$ for sGW and PGW and $\epsilon=1$ for UGW. With a small enough cutoff $\rho_s$, sGW is able to reconstruct the correct coupling between the two datasets. As $\rho_s$ increases, more connections are created and the symmetry underlying the data would cause reversed coupling for some parts of the data. In these experiments, PGW and UGW preserve the local geometry and global topology well. Since there are no hard constraints specifying upper bounds of distance changes after coupling, they are less sensitive to large-scale geometric configurations causing large-scale shifts compared to the ground truth coupling (Fig.\ref{fig:bioexample_4_comparison}).

\FloatBarrier
\section{Conclusions and Discussions}\label{sec:conclusion}
In this paper, we propose an sGW model framework to handle the matching problem between two datasets $\mathcal{X} = \{\mathbf{x}_i\}_i$ and $\mathcal{Y} = \{\mathbf{y}_j\}_j$ in distinct underlying metric spaces. In this model, we avoid the correspondence between $(\mathbf{x}_i, \mathbf{x}_j)$ and $(\mathbf{y}_k, \mathbf{y}_l)$ when the difference between $\|\mathbf{x}_i - \mathbf{x}_j\|$ and $\|\mathbf{y}_k - \mathbf{y}_l\|$ exceeds certain threshold. This is achieved by introducing an $\infty$-pattern in the tensor cost $\mathcal{M}$, and then transporting the most possible mass in $\mathbf{P}$ with minimal cost.

To apply sGW to large-scale practical problems, we design an sGW solver based on a heuristic minimum vertex cover algorithm, the Mirror-C descent algorithm, and the sOT algorithm. We also prove that the sequence $\{\mathbf{P}^{(k)}\}_{k\ge 0}$ generated by the sGW solver converges to a stationary point of the non-convex sGW objective provided that the step size $\eta$ (or $\Delta t$) is sufficiently small in the Mirror-C descent algorithm. Numerical experiments indicate that a relatively large step size accelerates the convergence and produces similar optimal coupling when the threshold $\rho$ in the tensor cost is small. Finally, by comparing sGW with other GW variants on synthetic data and practical single-cell RNA sequencing data, we highlight its distinctive advantage in overseeing the optimal matching.

The mathematical and numerical results from our study open up several future directions. Firstly, one may examine the relationship between sGW and PGW using the dummy point modification on the tensor cost instead of two distance matrices, particularly when the tensor cost possesses an $\infty$-pattern. Secondly, when the distributions in the two metric spaces are nonuniform, one needs to find weighted MVCs in order to describe the $0$-pattern condition in $\mathbf{P}$, making the characterization of the weighted graph critical. Besides, since it is computationally intractable to find MVC for a large graph within the sGW framework, we need to come up with an algorithm that can generate an approximate MVC yet is still efficient even for a large graph. Lastly, we may consider some accelerated iterative methods for the non-convex minimization, such as the Nesterov accelerated proximal gradient method.


\section*{Acknowledgements}

Z. Cang's work is supported by an NSF grant DMS-2151934. Y. Zhao's work is supported by a grant from the Simons Foundation through Grant No. 357963 and NSF grant DMS-2142500.

\appendix
\noindent
\section*{Appendix}
\section{Algorithm by Mirror Descent} 
In this appendix, we explain that the Mirror descent method can also generate a similar iteration as Eq. (\ref{eqn:MirrorC_01}) derived from Mirror-C descent. Additionally, the KL Mirror descent method is identical to the combined KL gradient and KL projection proposed in \cite{peyre2016gromov}.

Indeed, using the same notation as in (\ref{eqn:noncovex_objective}),
\begin{align}
\min_{\mathbf{P}\in \mathbb{R}^{n\times m}_+  } \mathcal{J}(\mathbf{P})+\mathcal{K}(\mathbf{P})
\end{align}
where
\begin{align*}
\mathcal{J}(\mathbf{P}) &= \frac{1}{2}\langle \mathcal{M}, \mathbf{P} \otimes \mathbf{P}\rangle_F,\\
\mathcal{K}(\mathbf{P}) &= -\epsilon H(\mathbf{P}) + \iota_{\mathbf{U}(\leq \mathbf{a}, \leq \mathbf{b}) \cap \mathcal{T}_\mathcal{M}^*} (\mathbf{P}) + \gamma \left( \| \mathbf{a}-\mathbf{P} \textbf{1}  \|_{1} + \| \mathbf{b}-\mathbf{P}^\mathrm{T} \textbf{1} \|_{1} \right) \\
& = \tilde{\mathcal{K}}(\mathbf{P}) + \iota_{\mathbf{U}(\leq \mathbf{a}, \leq \mathbf{b}) \cap \mathcal{T}_\mathcal{M}^*}(\mathbf{P}),
\end{align*}
and applying the KL Mirror descent method directly onto $\mathcal{J}$ and $\mathcal{K}$, we obtain that
    \begin{align*}
        \mathbf{P}^{(k+1)} & = \underset{\mathbf{P}\in \mathbf{U}(\leq \mathbf{a}, \leq \mathbf{b}) \cap \mathcal{T}_\mathcal{M}^*}{\mathrm{argmin}}\left\{ \left\langle \nabla \mathcal{J}(\mathbf{P}^{(k)}) + \nabla \tilde{\mathcal{K}}(\mathbf{P}^{(k)}), \mathbf{P} \right\rangle + \frac{1}{\Delta t'} \text{KL}(\mathbf{P}|\mathbf{P}^{(k)}) \right\}.
    \end{align*}
Note that $\langle \nabla \tilde{\mathcal{K}}(\mathbf{P}^{(k)}), \mathbf{P} \rangle = \langle \epsilon \log \mathbf{P}^{(k)} + 2\gamma \mathbf{1}, \mathbf{P} \rangle$, it follows that
    \begin{align*}
        \mathbf{P}^{(k+1)} & = \underset{\mathbf{P}\in \mathbf{U}(\leq \mathbf{a}, \leq \mathbf{b}) \cap \mathcal{T}_\mathcal{M}^*}{\mathrm{argmin}}\Big\{ \left\langle \nabla \mathcal{J}(\mathbf{P}^{(k)}), \mathbf{P} \right\rangle
        + \epsilon \langle \log \mathbf{P}^{(k)}, \mathbf{P} \rangle \\
        & \hspace{1.2in} + \gamma \left( \| \mathbf{a}-\mathbf{P} \textbf{1}  \|_{1} + \| \mathbf{b}-\mathbf{P}^\mathrm{T} \textbf{1} \|_{1} \right)
        + \frac{1}{\Delta t'} \text{KL}(\mathbf{P}|\mathbf{P}^{(k)}) \Big\}.
    \end{align*}
Now since $\nabla \mathcal{J}(\mathbf{P}) = \mathcal{M}\circ \mathbf{P}$, we further have that
    \begin{align*}
        \mathbf{P}^{(k+1)} & = \underset{\mathbf{P}\in \mathbf{U}(\leq \mathbf{a}, \leq \mathbf{b}) \cap \mathcal{T}_\mathcal{M}^*}{\mathrm{argmin}}
        \Big\{ 
        \Delta t' 
        \left\langle 
        \mathcal{M}\circ \mathbf{P}^{(k)}, \mathbf{P} 
        \right\rangle 
        + \epsilon \Delta t' \langle \log \mathbf{P}^{(k)}, \mathbf{P} \rangle \\
        & \hspace{1.2in} +  \gamma \Delta t' \left( \| \mathbf{a}-\mathbf{P} \textbf{1}  \|_{1} + \| \mathbf{b}-\mathbf{P}^\mathrm{T} \textbf{1} \|_{1} \right)  + \big(-H(\mathbf{P}) - \langle \mathbf{P}, \log \mathbf{P}^{(k)} \rangle \big) 
        \Big\}.    
    \end{align*}
    Combining the two terms of $\langle \log \mathbf{P}^{(k)}, \mathbf{P} \rangle$ and noting that $-H(\mathbf{P}) - \langle \mathbf{P}, \mathbf{Q} \rangle = \text{KL}(\mathbf{P}|\exp(\mathbf{Q}))$ for a generic $\mathbf{Q}$, it yields the KL Mirror iteration:    
\begin{align}
    \mathbf{P}^{(k+1)} &= \underset{\mathbf{P}\in \mathcal{T}_\mathcal{M}^*}{\mathrm{argmin}} \  \text{KL}(\mathbf{P}|\mathbf{W}^{(k)}) + h_1(\mathbf{P1}) + h_2(\mathbf{P}^{\top}\mathbf{1}) \label{eqn:Mirror_01} \\
    &= \underset{\mathbf{P}\in \mathcal{T}_\mathcal{M}^*}{\mathrm{argmin}} \  \text{KL} \left(\mathbf{P} \Big| \left[e^{-\mathcal{M}\circ\mathbf{P}^{(k)}/\epsilon}\right]^{\epsilon\Delta t'} \odot \left[\mathbf{P}^{(n)}\right]^{(1-\epsilon\Delta t')} \right) + h_1(\mathbf{P1}) + h_2(\mathbf{P}^{\top}\mathbf{1}). \label{eqn:Mirror_02}
\end{align}
where
\begin{align*}
    &\mathbf{W}^{(k)} = \exp \left( - \Delta t' \mathcal{M}\circ \mathbf{P}^{(k)} + (1-\epsilon\Delta t') \log \mathbf{P}^{(k)} \right), \\
    &h_1(\mathbf{P1}) = \Delta t' \gamma \|\mathbf{a}-\mathbf{P1}\|_1 + \iota_{[0,\mathbf{a}]}(\mathbf{P1}),\\
    &h_2(\mathbf{P}^{\top}\mathbf{1}) = \Delta t' \gamma \|\mathbf{b}-\mathbf{P}^{\top}\mathbf{1}\|_1 + \iota_{[0,\mathbf{b}]}(\mathbf{P}^{\top}\mathbf{1}).
\end{align*}

It is evident that when taking $\epsilon\Delta t'$ in KL Mirror descent (\ref{eqn:Mirror_02}) equal to $\frac{\epsilon\Delta t}{1+\epsilon\Delta t}$ in the KL Mirror-C descent (\ref{eqn:MirrorC_02}), namely, $\frac{1}{\Delta t'} = \frac{1}{\Delta t} + \epsilon$, KL Mirror descent and KL Mirror-C descent are equivalent. Additionally, when taking $\epsilon \Delta t' = 1$, the Mirror descent becomes identical to KL the projected gradient descent proposed in Proposition 2 of \cite{peyre2016gromov}.


\end{document}